\documentclass[12pt,letterpaper]{amsart}

\usepackage{accents}
\newtheorem{theorem}{Theorem}[section]
\newtheorem{lemma}[theorem]{Lemma}
\theoremstyle{definition}
\newtheorem{definition}[theorem]{Definition}

\newtheorem{corollary}[theorem]{Corollary}
\newtheorem{proposition}[theorem]{Proposition}
\theoremstyle{remark}
\newtheorem{remark}[theorem]{{\bf Remark}}

\newtheorem{claim}[theorem]{{\bf Claim}}

\numberwithin{equation}{section}

\usepackage{color}

\begin{document}

\def\C{\mathbb C}
\def\R{\mathbb R}
\def\X{\mathbb X}
\def\Z{\mathbb Z}
\def\Y{\mathbb Y}
\def\Z{\mathbb Z}
\def\N{\mathbb N}
\def\cal{\mathcal}
\def\cD{\cal D}
\def\tD{\tilde{{\cal D}}}
\def\F{\cal F}
\def\tf{\tilde{f}}
\def\tg{\tilde{g}}
\def\tu{\tilde{u}}

\def\cal{\mathcal}
\def\b{\mathcal B}
\def\c{\mathcal C}
\def\cc{\mathbb C}
\def\x{\mathbb X}
\def\r{\mathbb R}
\def\T{\mathbb T}
\def\uu{(U(t,s))_{t\ge s}}
\def\vv{(V(t,s))_{t\ge s}}
\def\xx{(X(t,s))_{t\ge s}}
\def\yy{(Y(t,s))_{t\ge s}}
\def\zz{(Z(t,s))_{t\ge s}}
\def\ss{(S(t))_{t\ge 0}}
\def\tt{(T(t,s))_{t\ge s}}
\def\rr{(R(t))_{t\ge 0}}
%% end of Preamble
\title[Traveling waves]{\textbf{TRAVELING WAVES IN reaction-diffusion EQUATIONS WITH General Delays}}

\author[Barker]{William Barker}
\address{Department of Mathematics and Statistics, University of Arkansas at Little Rock, \\
	2801 
	S University Ave, Little Rock, AR 72204. USA\\
	Email: wkbarker@ualr.edu}
 
\author[N.V. Minh]{Nguyen Van Minh}
\address{Department of Mathematics and Statistics, University of Arkansas at Little Rock, \\
	2801 
	S University Ave, Little Rock, AR 72204. USA\\
Email: mvnguyen1@ualr.edu}

\thanks{}

\date{\today}
\subjclass[2000]{Primary: 35C07 ; Secondary: 35K57 }
\keywords{Traveling waves; reaction-diffusion equations with delay}

\begin{abstract} We study the existence of traveling waves of reaction-diffusion equations of the form $\partial u(x,t)/\partial t = \Delta u(x,t-\tau_1)+f(u(x,t),u(x,t-\tau_2))$, where $\tau_1,\tau_2$ are positive constants. This spatial model could be applied to biological populations whose dynamics depend on the rates of migration with past history in addition to some delays in the rates of reproduction and death of individuals at  the same location.	
	We extend the monotone iteration method to systems that satisfy typical monotone conditions by thoroughly studying the Green function of the functional equation $x''(t)-ax'(t+r)-bx(t+r)=f(t)$, where $a\not=0, b>0$. In the framework of the monotone iteration method that is developed based on this result, upper and lower solutions are found for Fisher-KPP and Belousov-Zhabotinski equations to show that traveling waves exist for these equations when delays $r_1,r_2$  are small. The obtained results appear to be new.
\end{abstract}
\maketitle

\section{Introduction}

When justifying for the use of the reaction- diffusion model for interacting particle systems in population dynamics in spatial ecology  Durrett-Levin (\cite{durlev}) and Cantrell-Cosner (\cite[p. 25]{cancos}) mentioned two factors that affect the dynamics of the populations: (i) random migration and (ii) reproduction and death rates that depend on the number of individuals at the same location, or in the local neighborhood of that location. The random migration factor leads to the diffusion term $d\Delta u$, where $d$ is a positive constant, while the other factor leads to the "usual" reaction term as in the non-spatial model. If these factors instantly affect the dynamics of the population, then a reaction-diffusion model for the systems can be derived as (after a normalization of coefficients)
$$
\frac{\partial u(x,t)}{\partial t} =d\Delta u(x,t)+(a-bu(x,t))u(x,t),
$$
where $u(x,t)$ is the density of the population at location $x$ at time $t$. The reaction-diffusion model will be more realistic if we assume that the dynamics of the population depends not only on the instantaneous rates of change, but on a past history of (i) the random migration and (ii)  reproduction and death of the individuals at the same location.
With this assumption a  typical model taking into account of this delay effect considered in our paper may look like
\begin{equation}\label{fkpp}
\frac{\partial u(x,t)}{\partial t} =d\Delta u(x,t-\tau_1)+u(x,t-\tau_2)(1-u(x,t)),
\end{equation}
where $\tau_1,\tau_2$ are given positive constants. When $\tau_1=\tau_2=0$, the existence of traveling wave solutions and their stability in the above equation that becomes a Fisher-KPP equation
\begin{equation*}%\label{fkpp}
\frac{\partial u(x,t)}{\partial t} =d\Delta u(x,t)+u(x,t)(1-u(x,t)),
\end{equation*}
where $x\in \R$, 
are very well studied. The reader can find a complete account of the results and concepts in the classical monograph \cite{fif}. A rather more general and complete account of the existence and properties of traveling waves for parabolic equations can be found in \cite{volvolvol}.

A traveling wave solution $u(x,t)$ of Eq.(\ref{fkpp}) is defined as a twice continuously differentiable numerical function $\phi $ such that $u(x,t)=\phi (x\pm ct)$ for all $x,t \in \R$, where $c$ is a constant, and 
\begin{equation}\label{boundary}
\lim_{t\to -\infty} \phi(t)=0, \ \ \lim_{t\to\infty} \phi(t)=1.
\end{equation}
Such functions $\phi$ are often called the traveling wave front, where $c$ is the wave speed. The search for such functions $\phi$ is carried out by investigating the existence of solutions to an ordinary differential equation after the invariant substitution $u(x,t)=\phi (x\pm ct)$. Namely, one studies special solutions of the equation
\begin{equation}
\pm c\phi '(\xi )= \phi ''(\xi )+\phi (\xi )(1-\phi (\xi ))
\end{equation}
that satisfy (\ref{boundary}). The first attempt to study the existence and properties of traveling waves in reaction-diffusion equations with delay in reaction term
was made in \cite{sch}. Subsequently, a numerous amount of work has been done to study the affect of delay on the existence of traveling waves as well as the properties in diffusion-reaction systems, see e.g. \cite{fanshashi}, \cite{liumeiyan}, \cite{ma}, \cite{vol} \cite{wuzou} and the references there in for a few. It is noticed that the delay in the reaction-diffusion models that were considered so far was incorporated in the reaction term only. That leaves open a question as what happens if there is a delay in diffusion term of the models. It is the purpose of this paper is to study the existence of traveling waves in reaction-diffusion models when delay appears in both reaction as well as diffusion terms that satisfy some typical monotone conditions.  We will extend the method of monotone iteration to prove the existence of traveling waves to the models with delay appearing in both diffusion and reaction terms. For an account of this method the reader may see, e.g. \cite{boumin,ma,wuzou}. One of the very first and crucial step in this method is to study the bounded solutions whose existence is established by the Perron Theorem for functional equations. In fact, we need to know if the only bounded solution to the equation
\begin{equation}\label{fde1}
x''(t)-ax'(t+r)-bx(t+r)=f(t)
\end{equation}
is negative whenever the function $f(t)$ is positive, where $a>0,b>0$, $r=\tau_1 c$, $c$ is the wave speed of the traveling wave. When $\tau_1=0$ this equation is an ODE and the Green function can be found explicitly, so the answer to the question of negativity or positivity of the only bounded solution can be answered easily. The problem is much more complicated when $\tau_1>0$ as there is no explicit expression for the Green function. The existence of the latter is very well known in the theory of Functional Differential Equations. A great deal of this work will be devoted to the study of the positivity of bounded solutions. We will use the theory developed by Mallet-Paret in \cite{mal} combined with some standard techniques of Complex Functions to show that for sufficiently small $\tau_1>0$ the characteristic equation of Eq.(\ref{fde1}) has no root on the imaginary axis so the Green function exists and hence is negative. This allows us to extend the monotone iteration method to prove the existence of traveling waves for reaction-diffusion equations with delay appearing in both diffusion and reaction terms.

\medskip
This paper will be organized as follows: Section \ref{linear} will be devoted to the study of linear functional equations of the form (\ref{fde1}). We will give an answer to the question as if Eq.(\ref{fde1}) has a unique bounded solution if $f(\cdot )$ is a given bounded and continuous function. Then, we study the negativity of the Green function in Section \ref{positivity}. Theorem \ref{the 1} is the key for the construction of the monotone iteration method. In Section \ref{monotone iteration} we prove Theorem \ref{the main2} which summarizes all of the construction of the monotone iteration method. Construction of upper and lower solutions to the wave equations of particular models is always a nontrivial step in the applications of the monotone iteration method. We show that for Fisher-KPP and Belousov-Zhabotinski upper and lower solutions can be found explicitly.

\medskip
To the best of our knowledge this paper is the first study of traveling waves in reaction-diffusion equations with delay appearing in both the diffusion and reaction terms. The treatment of linear functional equations in Sections \ref{positivity} would have of independent interest. The results obtained in this paper appear to be new.

\section{Preliminaries and notations}
\subsection{Notations}
In this paper we will use some standard notations such as $\R,\C$ standing for the fields of reals and complex numbers. $\Re z$ and $\Im z$ denote the real part and imaginary part of a complex number $z$. The space of all bounded and continuous functions from $\R  \to \R^n$ is denoted by $BC(\R,\R^n)$ which is equipped with the sup-norm $\| f\| := \sup_{t\in\R} \| f(t)\|$. $BC^k(\R,\R^n)$ stands for the space of all $k$-time continuously differentiable functions $\R\to\R^n$ such that all derivatives up to order $k$ are bounded. If the boundedness is dropped from the above function spaces we will simply denote them by $C(\R,\R^n)$ and $C^k(\R,\R^n)$. We will use the natural order in $BC(\R,\R^n)$ that is defined as follows: 
For $f,g\in BC(\R,\R^n)$ we say that $f\le g$ if and only if $f(t)\le g(t)$ for all $t\in \R$, and we will say that $f< g$ if $f(t)\le g(t)$ for all $t\in \R$, and $f(t)\not= g(t)$ for all $t\in \R$.
An operator $P$ acting in a subset $S$ of $BC(\R,\R^n)$ is monotone if it preserves the natural order in the subset. That is, $Pf \le Pg$ whenever $f\le g$ and $f,g\in BC(\R,\R^n)$. Note that if $S$ is the whole space $BC(\R,\R^n)$ and $P$ is linear, then $P$ is monotone if and only if it maps a positive function into a positive function, where $f\in BC(\R,\R^n)$ is positive if $f(t)\ge 0$ for all $t\in \R$. A constant function $f(t)=\alpha$ for all $t\in \R$ will be denoted by $\hat \alpha$.

\medskip
We will write $L^p$ for the space $L^p(\R,\C^n)$ of $L^p$ vector valued functions on the line when $n$ is a clear given positive integer. We assume that $1\le p\le \infty$ and denote the space
$$
W^{1,p}=\{ f\in L^p | \ f \ \mbox{is absolutely continuous, and}\ f'\in L^p\}.
$$
Recall that the continuous embedding
$$
W^{1,p} \subset L^\infty , \ 1\le p\le \infty .
$$

\subsection{Rouch\' e's Theorem} \label{the rouche}
Let $A$ be an open subset of $\C$, $f$ and $g$ are two analytic functions on $A$. A piecewise continuously differentiable function $\gamma :[a,b]   \to \C$ such that $\gamma (a)=\gamma (b)$ is called a circuit. The following theorem (\cite[Rouch\' e's Theorem, p. 247]{die}) will be used later on:
\begin{theorem}
let $A\subset \C$ be a simply connected domain, $f,g$ two analytic complex valued functions in $A$. Let $T$ be the (at most denumerable) set of zeros of $f$, $T'$ the set of zeros of $f+g$ in $A$, $\gamma$ a circuit in $A- T$, defined on an interval $I$. Then, if $|g(z)| < |f(z)|$ in $\gamma (I)$, the function $f+g$ has no zeros on $\gamma (I)$, and  
\begin{equation}
\sum_{a\in T} j(a;\gamma ) \omega (a;f) =\sum_{b\in T'} j(b;\gamma )\omega (b;f+g),
\end{equation}
where $j(a;\gamma)$ is the index of $\gamma$ with respect to $a$, and $\omega (a,f)$ is the multiplicity of the zeros of $f$ at $a$.
\end{theorem}
As a consequence of the theorem, if $\gamma$ is a simple closed circuit that encircles $T$, then, $j(a,\gamma )=1$ for all $a$ inside $\gamma (I)$ and Rouch\' e Theorem claims that the total zeros (counting multiplicities) of $f$ and $f+g$ inside the circuit $\gamma$ are the same.

\subsection{Perron Theorem for functional differential equations }
This subsection is concerned with well known results (Perron Theorem) on the existence of a unique bounded solutions to mixed differential equations of the form 
\begin{equation}\label{DDE}
x''(t)-ax'(t+r)-bx(t+r)=f(t),
\end{equation}
$a\not=0,b>0$, if $r>0$ the equations are advanced, and if $r<0$ they are delayed equations, $f(\cdot )$ is in $BC(\R,\R)$ or $L^\infty(\R,\R)$.

\medskip
We recall the Perron Theorem in an extended form as below. Consider functional differential equations of the form
\begin{equation}\label{fe}
x'(t)=\sum_{j=1}^N A_jx(t+r_j)+h(t), \ x(t)\in \R^n,
\end{equation}
where $r_j$ are positive constants $j=1,2,\cdots, N$, $h\in L^p(\R,\R^n)$, $A_j$ are $n\times n$-matrices. Let us denote the characteristic equation associated with Eq.(\ref{fe}) by
$$
\Delta (\lambda ):= \lambda I - \sum_{j=1}^N e^{\lambda r_j}A_j.
$$
When $\det \Delta (\lambda )\not =0$ for all $\lambda \in i\R$, Eq.(\ref{fe}) is called hyperbolic. Let us define the operator ${\cal L}$ as follows
$$
[{\cal L} x ](t) = x'(t)- \sum_{j=1}^N A_jx(t+r_j),
$$
where $x(\cdot )\in W^{1,p}$ for some $1\le p\le \infty$. The following theorem will be used in our construction of the monotone iteration later to show the existence of traveling waves:
\begin{theorem}\label{the per} 
Assume that Eq.(\ref{fe}) is hyperbolic. Then, the operator ${\cal L}$ is an isomorphism from $W^{1,p}$ onto $L^p$ for $1\le p\le \infty$ with inverse given by convolution
$$
({\cal L}^{-1}h)(\xi) =(G*h)(\xi ) =\int^\infty_{-\infty} G(\xi -s)h(s )ds,
$$
where
$$
G(\xi) =\frac{1}{2\pi} \int^\infty_{-\infty} e^{i\xi\eta}\Delta (i\eta)^{-1} d\eta 
$$
that enjoys the estimate
$$
\| G(\xi )\| \le Ke^{-\alpha |\xi|}
$$
for some $K>0$ and $\alpha >0$. In particular, for each $h\in L^p$ there exists a unique solution $x(\cdot )={\cal L}^{-1} h \in W^{1,p}$ to the inhomogeneous equation (\ref{fe}).
\end{theorem}
A special case of this theorem when Eq.(\ref{fe}) becomes Eq(\ref{DDE}) is the following. Note that when $f$ is continuous the theorem is the classical Perron Theorem in which all solutions mentioned in the statements are classical solutions (see \cite{murnaimin,pru}).

\begin{theorem}\label{the per2}
Consider a ordinary functional differential equation, 
\begin{equation}
x''(t)-ax'(t+r)-bx(t+r)=f(t), 
\end{equation}
where $a\not=0,b>0, r\in \R$ and $f\in BC (\R ,\R)$ with associated characteristic function $P (\lambda)=\lambda^2-a\lambda e^{r\lambda}-be^{r\lambda}.$ If $ P (i\xi)\neq 0$ for all $\xi\in \R,$ then the differential equation has a unique bounded solution. Moreover, the solution
is given by 
\[x_{f}(t)=(G\ast f)(t)=\int_{-\infty}^{\infty}G(t-s)f(s)ds,\]
where $G(t)$ is a Green's function that exponentially decays to 0 as $|t|\to \infty.$ In particular, if $f\in L^\infty$, then there exists a unique $x(\cdot )$ such that $x(\cdot )$, $x'(\cdot )$ are absolutely continuous, $x''(\cdot )\in L^\infty$.
\end{theorem}
\begin{remark}
As shown in \cite{boumin}, removing the absolute continuity of $x'(\cdot)$ would be wrong. This will affect the way we construct upper and lower solutions later in Section \ref{applications}.
\end{remark}

\section{Qualitative analysis of characteristic equations of linear functional differential equations}\label{linear}
In this section we will develop qualitative results for second order linear functional differential equations of mixed type. We will incorporate functional arguments in both the $x(\cdot)$ and $x'(\cdot)$ terms. We will build the aforementioned results based off of qualitative results of non delay ordinary differential equations.

The second order non delay equation 
\begin{equation} \label{NDE}
x''(t)-ax'(t)-bx(t)=f(t),
\end{equation} 
where $a\not=0,b>0$,  and $f(t)$ is a bounded continuous function for $t\in \R$ has a unique bounded solution for $t\in \R.$ This is a direct result of the Perron Theorem. To see this, the characteristic equation is \[\Delta_1(\lambda)=\lambda^2-a\lambda -b.\] The roots of the characteristic equation are 
\begin{equation*}
    \lambda_1=\frac{a+\sqrt{a^2+4b}}{2}>0,
    \lambda_2=\frac{a-\sqrt{a^2+4b}}{2}<0. 
\end{equation*}
The solution of the differential equation is 
\[x_f(t)=\int_{-\infty}^{\infty}G(t-s)f(s)ds.\]

A simple computation shows that
\begin{equation}
x_f(t)= \frac{1}{\lambda_2-\lambda_1} \left( \int^t_{-\infty} e^{\lambda _2 (t-s)}f(s)ds +\int^{+\infty}_t e^{\lambda_1 (t-s)}f(s)ds \right) ,
\end{equation}
and thus,
\begin{equation}
G(\xi )=\begin{cases}
e^{\lambda_1 \xi }/(\lambda_2-\lambda_1), \ \mbox{if} \ \xi < 0,\\  
e^{\lambda_2 \xi }/(\lambda_2-\lambda_1) , \ \mbox{if} \ \xi \ge  0.
\end{cases}
\end{equation}
The Green's function is obviously negative, so for all $t\in \R$
$$
x_f(t) \le 0, \ \mbox{if} \ f(t)\ge 0 .
$$
\subsection{Characteristic roots of functional differential equations}
This subsection is concerned with the formulation and proof for some results, that are important  for our later use, for equations of the form
\begin{equation}\label{inho}
x''(t) -ax'(t+r)-bx(t+r)=f(t),
\end{equation}
where $a\not=0, b>0$ and $r\in\R$ is assumed to be a small. The characteristic equation of the aforementioned equation is 
\begin{equation}\label{c}
z^2-aze^{rz}-be^{rz}=0.
\end{equation}

There are often issues with solving equations similar to (\ref{c}), as they are no longer polynomials, but are transcendental functions. These "exponential" type polynomials have been studied in \cite{cookegrossman,hale,Wei1,Wei2}. Often, transcendental polynomials can be studied using methods found in complex analysis. 
In fact, a crucial piece of complex analysis, which is attributed to the French mathematician Eugène Rouché,  will be employed to prove the following lemma. This is Rouché's mentioned in section 2.

\begin{proposition}\label{pro 1}
For given $a\not=0,b>0$, the following assertions are true:
\begin{enumerate}
\item For every $r$ Eq.(\ref{c}) has no root on the imaginary axis;
\item For every $r \le 0$ ($r\ge 0$, respectively), Eq. (\ref{c}) has only one single root in the right half of the complex plane (Eq. (\ref{c}) has only one single root in the left half of the complex plane, respectively), and the root has continuous dependence on $r$.
\item For sufficiently small $r\le 0$ ($r\ge 0$, respectively) there exists only a unique single root (so it is real!) in the strip $\{ z\in\C | \  2\lambda_2 \le \Re z \le 0\}$  (there exists only a single root in the strip $\{ z\in\C | \   0 \le \Re z \le 2\lambda_1\}$, respectively), and the root has continuous dependence on $r$.
\end{enumerate}
\end{proposition}
\begin{proof}
Consider the functions $\alpha (z)$ and $\beta (z)$ defined as
$$
\alpha (z):= (b+az) (1-e^{rz} ); \quad  \beta (z):= z^2-az-b .
$$
Then,
$$
\gamma (z):= \alpha (z)+\beta (z) = z^2 -aze^{rz}-be^{rz}.
$$
Part (i): We will show that  $|\alpha (z)| < \beta (z)|$ for all $z$ along the imaginary line $i\R$, so Eq.(\ref{c}) cannot have any root on $i\R$.
Consider the function $\alpha (z)$ on the line segment $\{ z\in \C | \ z=i\xi, -R \le \xi \le R\}$, where $R>0$ is a positive number. We have
\begin{eqnarray}
|\alpha (i\xi ) | &=& |b+ai\xi|\cdot |1-e^{ir\xi}| \nonumber \\
&=& \sqrt{b^2+(a\xi)^2}  \sqrt { (1-\cos (r\xi ))^2+\sin^2(r\xi )} \nonumber \\
&=&  \sqrt{b^2+(a\xi)^2} \sqrt { 4\sin^4 (r\xi /2)+4\cos^2(r\xi/2)\sin^2(r\xi /2)}  \nonumber \\
&=& 2  \sqrt{b^2+(a\xi)^2}  |\sin (r\xi /2)) | . \label {hc}
\end{eqnarray}
On the other hand,
\begin{equation}
|\beta (i\xi )|= |-\xi^2 +ai\xi -b| =\sqrt{(\xi^2+b)^2+(a\xi)^2}  .
\end{equation}
Therefore, for every $\xi\in\R$,
\begin{equation}\label{l}
\frac{|\alpha (i\xi )|}{|\beta (i\xi )|} = \frac{\sqrt{b^2+(a\xi)^2} }{\sqrt{(\xi^2+b)^2+(a\xi)^2}}|\sin (r\xi /2)| <1.
\end{equation}
 This shows, by Rouché's theorem that Eq.(\ref{c}) has no root on the imaginary axis.

\medskip
Part (ii): 
We are going to show that if $r\ge 0$,  $\gamma (z)$ has only one root in the half complex plane $\{ z\in \C| \ \Re z <0\}$ by 
a careful application of Rouche's Theorem \ref{the rouche}. To that end, we will consider a positively oriented contour $C$ consisting of the line segment $\{ z\in \C | \ z=i\xi, -R \le \xi \le R\}$ and the half circle $\{ z\in \C | \ z=Re^{i\theta }, -\pi/2 \le \theta \le \pi /2 \}$, where $R$ is a given large positive number. We 
will choose the radius $R$ sufficiently large later so that $|\alpha (z)| < |\beta (z)|$ for all $z$ in the half circle $\{ z\in \C | \ z=Re^{i\theta }, -\pi/2 \le \theta \le \pi /2 \}$. 
In fact, since a simple computation shows that if $z=Re^{i \theta}$, where $-\pi/2 \le \theta \le \pi /2 $, then, $\cos (\theta )\ge 0$, so
\begin{align*}
|1-e^{rz}| &= |1-e^{rR(\cos (\theta)+i\sin(\theta)}| \\
&=  |1-e^{rR\cos (\theta)}\cdot  e^{irR\sin(\theta)}|\\
& \le1+ e^{rR\cos (\theta)}\\
& \le  2.
\end{align*}
Therefore, for $z=Re^{i\theta }, -\pi/2 \le \theta \le \pi/2 $, we have
\begin{align*}
 \limsup_{R \to \infty} \frac{|\alpha( z) |}{|\beta(z)|} 
&=   \limsup_{R \to \infty} \left| \frac{(b+az) (1-e^{rz} )}{ z^2-az-b}       \right|  \\
&\le   2 \lim_{R \to \infty}   \frac{|b+az|}{ |z^2-az-b|}  =0 .
\end{align*}
Finally, for a fixed sufficiently large $R>0$, we have 
\begin{equation}
|\alpha (z)| <   |\beta (z)|
\end{equation}
for all $z\in \{ z\in \C | \ z=Re^{i\theta }, -\pi/2 \le \theta \le \pi /2 \}$. By Part (i), $|\alpha (z)| <|\beta (z)|$ for all $z$ on the imaginary axis.
Hence, if $r\ge 0$, for sufficiently large $R>0$, we have
\begin{equation}
|\alpha (z)| < |\beta (z)|
\end{equation}
for each $z\in C$. Note that $\beta(z)$ has no root on $C$, so by the Rouch\' e Theorem \ref{the rouche}, inside the contour $C$ the number of roots (counting multiplicities) of Eq.(\ref{c}) is the same as of the quadratic equation $z^2-az-b=0$, that is, only a single root. As $R>0$ can be chosen to be any large number, this yields that on the left half plane there exists only one simple root of Eq.(\ref{c}).
Furthermore, the continuous dependence is due to the Implicit Function Theorem (see e.g.  \cite{die}). The case $r<0$ is treated in the same way.

\bigskip
Part (iii): We will estimate $|\alpha (z)|/|\beta(z)|$ on the boundary of the rectangle $\Xi$ defined as $\{ z\in \C | \ 2\lambda_2 \le \Re z\le 0, -R \le \Im z \le R\}$ for a positive number $R$. By Part (i), for sufficiently small $r<0$ the ratio $|\alpha (z)|/|\beta (z)| <1$ on the imaginary axis. For brevity, denote $\delta :=-2\lambda_2$ and consider the rectangle's edge $\{ z\in \C | \ z=-\delta +i\xi , -R \le \xi \le R\}$. On the line
$\{ z\in \C | z=-\delta +i\xi , \xi \in \R\}$
\begin{eqnarray}
|\alpha (-\delta+i\xi ) | &=& |b+a(-\delta+i\xi)|\cdot |1-e^{-r\delta+ir\xi}| \nonumber \\
&\le& \sqrt{(b-a\delta)^2+(a\xi)^2}  \sqrt { (1-e^{-r\delta}\cos (r\xi ))^2+e^{-2r\delta}\sin^2(r\xi )  } \nonumber \\
&\le&  
  \sqrt{(b-a\delta)^2+(a\xi)^2}  \sqrt{1-2e^{-r\delta}\cos(r\xi )+e^{-r\delta}} \nonumber \\
  &\le &  \sqrt{(b-a\delta)^2+(a\xi)^2} 2e^{-r\delta/2}.  \label {hc2}
\end{eqnarray}

Therefore, it is easily seen that
\begin{align*}
0\le \limsup_{|\xi|\to \infty} \frac{|\alpha (-\delta +i\xi )|}{|\beta (-\delta +i\xi )|} \le 2e^{r\delta/2} \lim_{|\xi| \to\infty}   
 \frac{ \sqrt{(b-a\delta)^2+(a\xi)^2} }{|(-\delta +i\xi)^2-a(-\delta+i\xi)-b|} =0.
\end{align*}
There exists a positive constant $K$ such that if $|\xi| \ge K$, then
$$
0\le \limsup_{|\xi|\to \infty} \frac{|\alpha (-\delta +i\xi )|}{|\beta (-\delta +i\xi )|} \le  \frac{1}{2}
$$
Since
$$
\lim_{r\to 0}  \sqrt{1-2e^{r\delta}\cos(r\xi )+e^{r\delta}} =0
$$
for sufficiently small  $r$, 
\begin{align}
\frac{|\alpha (-\delta+i\xi )|}{|\beta (-\delta +i\xi )|} <\frac{1}{2}
\end{align}
for all $|\xi| \le K$. This yields that for sufficiently small $r$, $|\alpha (z)|/|\beta (z)|<1/2$ on the edge 
$\{ z\in \C | \ z=\delta +i\xi , -R \le \xi \le R\}$. 

\medskip
On the remaining edges of the rectangle: $\{ z\in \C | \ z=\xi +iR, -\delta \le \xi \le 0\}$ and $\{ z\in \C | \ z=\xi -iR, -\delta \le \xi \le 0\}$, we have
\begin{align}
|\alpha (\xi +iR)| & = |b+a(\xi +iR)|\cdot |1-e^{r(\xi+iR)}| \nonumber\\
& \le \sqrt{(b+a\delta )^2 +(aR)^2} \cdot (1+e^{r\delta}).
\end{align}
Therefore, for sufficiently large $R$
\begin{equation}
\frac{|\alpha (\xi +iR)|}{|\beta(\xi+iR)|} \le \frac{1}{2}.
\end{equation}
By Rouch\' e Theorem \ref{the rouche}, in the rectangle there exists exactly one single root of the equation $\gamma (z)=0$ as the equation $\beta(z)=0$ has only one single root $\lambda_2$. Since the number $R$ can be any large positive number, this follows that in the strip mentioned above there exists only one single root. Next, by the Implicit Function Theorem, this only root depends continuously on $r$. The case $r\ge 0$ is treated similarly.
\end{proof}

Before proceeding we need the following claim:
\begin{claim}\label{cla 1}
Let $a\not=0$ and $b>0$ be any numbers. Then,
\begin{align}
a(a+\sqrt{a^2+4b})+2b &> 0,\label{1.17}\\
a(a-\sqrt{a^2+4b})+2b &>0 .\label{1.18}
\end{align}
\end{claim}
\begin{proof}
By assumption we have
\begin{align*}
 a^2(a^2+4b)& < a^4+4a^2b+4b^2 \\
 &= (a^2+2b)^2 ,
\end{align*}
so
\begin{align}\label{1.19}
 -(a^2+2b) &<  a\sqrt{a^2+4b} <a^2+2b.
\end{align}
The first inequality in (\ref{1.19}) yields (\ref{1.17}), and the second inequality in (\ref{1.19}) yields (\ref{1.18}).
\end{proof}
\begin{lemma}\label{lem 1.3}
Let's denote $\eta_1$ and $\eta_2$ as the only root of Eq. (\ref{c}) in the strip $\{ 0\le \Re z \le 2\lambda_1\}$ and $\{2\lambda_2 \le \Re z \le 0\}$, respectively. Then,
$$
\lim_{r\downarrow 0}  \frac{ d\eta _1 (r)}{dr}>0
$$
\end{lemma}
\begin{proof}
Consider the function
$$
f(r,z)=z^2-aze^{rz}-be^{rz}=0,
$$
where $r,z$ are in $\R$.
By the Implicit Function Theorem, near $\lambda_{1}$ the root $\eta_1$ as function of $r$ for small $|r|$, and 
\begin{align*}
\frac{dz}{dr} &=-\frac{\frac{\partial f(r,z)}{\partial r}}{\frac{\partial f(r,z)}{\partial z}} \\
& =\frac{az^2 e^{rz}+bze^{rz} }{ 2z-a(e^{rz}+rze^{rz})-bre^{rz}}\\
&= \frac{ze^{rz}(az+b) }{ 2z-a(e^{rz}+rze^{rz})-bre^{rz}}.
\end{align*}
Therefore, near $(r,z)=(0,\lambda_1)$
\begin{align}
\frac{d\eta_1 }{dr}
&= \frac{\eta_1e^{r\eta_1}(a\eta_1+b)}{2\eta_1 -a (e^{r\eta_1}+r\eta_1 e^{r\eta_1})-bre^{r\eta_1}}\nonumber .
\end{align}
Therefore, by (\ref{1.17})  $a\lambda_1+b >0$, so
$$
\lim_{r\downarrow 0}  \frac{ d\eta _1 (r)}{dr}=\frac{\lambda_1(a\lambda_1+b)}{2\lambda_1 -a}=\frac{\lambda_1(a\lambda_1+b)}{\sqrt{a^2+4b}} >0.
$$

\end{proof}
\begin{remark}
This lemma is important for us to prove Lemma \ref{lem1} below. In fact, this lemma shows that the strip $\{ 0 \le \Re z \le \Re \eta_1\}$ where  no root of the function $f(r,z)$  expands as $r$ increases from $0$, so it contains $\{ 0\le \Re z \le \lambda_1\}$.
\end{remark}
 
\section{Positive Solutions of Second Order Delay Equations}\label{positivity}
The sign of any solution to a given differential equation is of upmost importance. In fact, when there is no delay results are known about the sign of second order equations.This was discussed at the beginning of the previous section. In this section, novel results will be put forth for second order delay equations.The main result of this work can now be stated.

\bigskip
By Proposition \ref{pro 1}, for sufficiently small $r$ the characteristic equation has no root on the imaginary axis, so the inhomogeneous equation (\ref{inho}) has a unique bounded solution $x_f$ for each given bounded and continuous $f(\cdot )$.

\begin{theorem}\label{the 1}
For sufficiently small $r\ge 0$ let $G(t,r)$ be the Green Function of Eq.(\ref{inho}) such that
$x_f(t)=\int_{-\infty}^{\infty}G(t-s,r)f(s)ds$ is the unique bounded solution to Eq.(\ref{inho}), then $G(t,r)<0$ for all $t\in\R$. 
\end{theorem}
The proof will be the result of the following Lemma that will be proven hence forth and the Lagrange Mean Value Theorem.
\begin{lemma}\label{lem1}
Define $G(t,r)$, and $G(t)$ as the Green function for equations (\ref{DDE}) and (\ref{NDE}), respectively. Then, for $r>0$ 
\[\sup_{t<0}\left|\frac{\frac{\partial G(t,r)}{\partial r}}{G(t)}\right|< \infty.\]
\end{lemma}
\begin{proof}

\medskip
The Green's function for the delay solution is given by (see Mallet-Paret, \cite{mal,pru}).
\[G(t,r)=\frac{1}{2\pi}\int_{-\infty}^{\infty}\frac{e^{i\xi t}}{-\xi^2-ai\xi e^{i\xi r}-be^{i\xi r}}d\xi.\]
The aforementioned integral is absolutely convergent, so
\begin{align}
 \frac{\partial G(t,r)}{\partial r} &=\frac{1}{2\pi}\frac{\partial}{\partial r}  \int_{-\infty}^{\infty} \frac{e^{i\xi t}}{-\xi^2-ai\xi e^{i\xi r}-be^{i\xi r}}d\xi 
 \nonumber\\
 &=\frac{1}{2\pi}  \int_{-\infty}^{\infty}\frac{\partial}{\partial r} \left(\frac{e^{i\xi t}}{-\xi^2-ai\xi e^{i\xi r}-be^{i\xi r}}\right)d\xi \nonumber \\
  &= \frac{1}{2\pi}  \int_{-\infty}^{\infty} \frac{e^{i\xi (t+r)}\left(a\xi^2-b\xi i\right)}{\left(\xi^2+ai\xi e^{i\xi r}+be^{i\xi r}\right)^2} d\xi  \label{2.1}
 \end{align}

By Part (iii) Proposition \ref{pro 1} and Lemma \ref{lem 1.3} we can move the integral line down to the parallel line $\{ z=\xi -i \lambda_1, \xi \in \R\}$, so we have
\begin{align}
\left|  \frac{\partial G(t,r)}{\partial r} \right| &= \left| \frac{1}{2\pi}  \int_{-\infty}^{\infty} 
\frac{e^{i(\xi-i\lambda_1) (t+r)}\left(a(\xi-i\lambda_1)^2-b(\xi-i\lambda_1) i\right)}
{\left((\xi-i\lambda_1)^2+ai(\xi-i\lambda_1) e^{i(\xi-i\lambda_1) r}+be^{i(\xi-i\lambda_1) r}\right)^2} d\xi \right|   \nonumber\\
 &= \frac{e^{\lambda_1(t+r)}}{2\pi} \left| \int_{-\infty}^{\infty} \frac{e^{i\xi (t+r)}\left(a(\xi-i\lambda_1)^2-b(\xi-i\lambda_1|) i\right)}
 {\left((\xi-i\lambda_1)^2+ai(\xi-i|\lambda_1) e^{i(\xi-i\lambda_1) r}+be^{i(\xi-i\lambda_1) r}\right)^2} \right| d\xi \nonumber \\
 &\le K_0e^{\lambda_1t} , \label{2.3} 
 \end{align}
where 
$$
K_0:=\frac{e^{\lambda_1 r}}{2\pi}\int_{-\infty}^{\infty} \frac{\left|a(\xi+i\lambda_1)^2-b(\xi+i\lambda_1) i\right|}{\left|(\xi-i\lambda_1)^2+ai(\xi-i\lambda_1) e^{i(\xi-i\lambda_1) r}+be^{i(\xi-i\lambda_1) r}\right|^2} d\xi <\infty .
$$
Therefore, for $t<0$,
\begin{align}
\left|  \frac{\frac{\partial G(t,r)}{\partial r}}{G(t)}  \right| &\le K_1 .
\end{align}
\end{proof}

Now, we can prove Theorem \ref{the 1}.
\begin{proof}
The case $t<0$: In light of Lemma \ref{lem1}, the fact that the partial derivative of the Green's function is convergent, and the Lagrange Mean Value, we have the following estimate.
\begin{align*}\left|G(t,r)-G(t)\right|&\le \sup_{0<\omega<r}\left|\frac{\partial G(t,r)}{\partial r}\right|r\\
 \left|\frac{G(t,r)}{G(t)}-1\right|&\le M r
 \end{align*}
If we take $0<r<{1}/{M},$ then ${G(t,r)}/{G(t)}>0.$ This means, for sufficiently small $r>0$, the Green function $G(t,r)<0$ for all $t<0$. 

\bigskip
The case $t>0$: By Proposition \ref{pro 1}, in the upper half of the complex plane $\{ z\in\C :\ \Im z\ge 0\}$ the function
$\xi^2+ai\xi e^{i\xi r}+be^{i\xi r}$ has a unique imaginary root $-i\eta_2(r)$ that is a simple pole. As $t>0$ and $\Im \xi >0$, we have
\begin{align*}
\left| \frac{e^{i\xi t}}{-\xi^2-ai\xi e^{i\xi r}-be^{i\xi r}} \right| &=  \left| \frac{e^{-\Im \xi t}}{-\xi^2-ai\xi e^{i\xi r}-be^{i\xi r}} \right|\\
&\le  \frac{1}{|\xi^2+ai\xi e^{i\xi r}+be^{i\xi r}|}.
\end{align*}
Therefore, in the upper half plane using a half-circle centered at the origin with large radius, and then the Residue Theorem we can show that
\begin{align}
 G(t,r) &=\frac{1}{2\pi}\int_{-\infty}^{\infty}\frac{e^{i\xi t}}{-\xi^2-ai\xi e^{i\xi r}-be^{i\xi r}}d\xi \nonumber \\
 &=i Res (g; -i\eta_2)\nonumber\\
 &= i\lim_{z\to-i\eta_2}  \frac{e^{iz t}(z+i\eta_2)}{-z^2-aize^{iz r}-be^{iz r}},
\end{align}
where
$$
g(z,r):= \frac{e^{iz t}(z+i\eta_2)}{-z^2-aize^{iz r}-be^{iz r}}.
$$
Therefore,
\begin{align}
i \lim_{z\to - i\eta_2} (z+i\eta_2) g(z,r) &=i e^{i(-i\eta_2) t}   \lim_{z\to - i\eta_2}   \frac{z+i\eta_2}{-z^2-aize^{iz r}-be^{iz r}}  \nonumber \\
&= e^{\eta_2 t}  \lim_{z\to - i\eta_2}   i\frac{z+i\eta_2}{-z^2-aize^{iz r}-be^{iz r}}  .
\end{align}

Since $-i\eta_2$ is a simple pole of the function  $-\xi^2-ai\xi e^{i\xi r}-be^{i\xi r}$ the function
$$
h(z,r):= i\frac{z+i\eta_2}{-z^2-aize^{iz r}-be^{iz r}}
$$
is extendable holomorphically in the upper plane, so it is continuous in the variable $(z,r)$. Hence
\begin{align*}
 \lim_{r\to 0}   \lim_{z\to - i\eta_2}   h(z,r) &= \lim_{(r,z)\to (0,-i\lambda_2)}h(z,r)\\
 &=\lim_{z\to -i\lambda_2}\lim_{r\to 0} h(z,r)\\
 &=\lim_{z\to -i\lambda_2}i\frac{z+i\lambda_2}{-z^2-aiz-b}\\
 &=\lim_{z\to -i\lambda_2}i\frac{z+i\lambda_2}{-(z+i\lambda_1)(z+i\lambda_2)}\\
 &= \frac{i}{-(-i\lambda_2+i\lambda_1)}\\
 &=\frac{1}{ \lambda_2-\lambda_1} <0.
\end{align*}
This show that there is a sufficiently small $r$ (independent of $t$) such that the function $G(t,r) <0$ for all $t >0$. This completes the proof.

\end{proof}

\section{Monotone Iteration Method for Traveling Waves}\label{monotone iteration}
This section will focus on equations of the form 
\begin{equation}\label{RD1}
\frac{\partial u(x,t)}{\partial t}=D\frac{\partial^2 u(x,t-\tau_1)}{\partial x^2} +f(u_t),
\end{equation}
where $t\in\R, \tau_1>0,  x, u(x,t)\in \R, \ D>0, \ f:C\left([-\tau_2,0], \R\right)\to \R$ is continuous and $u_t(x)\in C\left([-\tau_2,0], \R\right),$ defined as 
\[u_t(x)=u(x, t+\theta), \ \theta\in [-\tau_2,0], \ t\ge 0, \ x\in \R.\]
Moreover, we will assume that $f$ is Lipschitz continuous and 
\[f(\hat 0)=f(\hat K)=0, \ \text{and} \ f(u)\neq 0, \ \hat 0<u< \hat K.\]
We are interested in traveling wave solutions of the form $u(x,t)=\phi(x+ct), \ c>0.$ This transformation gives 
\begin{align*}\label{RD2}
u(x,t)&=\phi(x+ct)\\
\frac{\partial u(x,t)}{\partial t}&= c\phi'(x+ct)\\
\frac{\partial^2 u(x,t-\tau_1)}{\partial^2 x}&=\phi''(x+c(t-\tau_1)).
\end{align*}
Setting $\xi=x+ct, r_1=c\tau_1,$ Eq. (\ref{RD1}) becomes the following ordinary differential equation
\begin{equation}\label{W1}
c\phi'(\xi)=D\phi''(\xi-r_1)+f_c(\phi_\xi ),    
\end{equation}
where
$f_{c}\in \X_c:= :C([-c\tau_2,0],\mathbb{R}^{n})\rightarrow\mathbb{R}$,
defined as
\[
f_{c}(\psi)=f(\psi^{c}),\quad\psi^{c}(\theta):=\psi(c\theta),\quad\theta
\in\lbrack-\tau_2,0].
\]
By a shift of variable we can reduce the equation to 
\begin{equation}\label{wave}
D\phi'' (\xi)-c\phi' (\xi+r_1)+f_{c}(\phi_{\xi+r_1})=0,\;\;\xi
\in\mathbb{R},
\end{equation}

\bigskip
The main purpose of this section is to look for solutions $\phi$ of
(\ref{wave}) in the following subset of
$C(\mathbb{R},\mathbb{R})$
\[
\Gamma:=\{\varphi\in
C(\mathbb{R},\mathbb{R}):\varphi\;\;\;\mbox{is
nondecreasing, and}\ \ \lim_{\xi\rightarrow-\infty}\varphi(\xi)=0
,\ \lim_{\xi\rightarrow+\infty}\varphi(\xi)={K}\} .
\]

\bigskip
We assume that 
there exists a positive number $\beta$ with such that
\begin{equation}\label{A}
f_{c}(\phi)-f_{c}(\psi)+\beta\lbrack\phi(0)-\psi(0)]\geq 0,
\end{equation}
for all $\phi , \psi \in \X_c$ such that $\phi \ge \psi$. Let us consider the linear operator ${\cal L}$ in $BC(\R,\R)$ defined as
\begin{equation}\label{oper L}
{\cal L}(\phi)(\xi):= D\phi'' (\xi)-c\phi' (\xi+r_1)-\beta \phi(\xi +r_1)=f(t),
\end{equation}
where $f\in BC(\R,\R )$, and 
\begin{equation}\label{A2}
H(\phi)(t)=f_{c}(\phi_{t+r_1})+\beta\phi(t+r_1),\;\;\;\;\;\;\;\;\phi\in C(\mathbb{R}%
,\mathbb{R}).
\end{equation}
By Proposition \ref{pro 1} and Theorem \ref{the per2} we can see that the operator ${\cal L}$ is invertible and by Theorem \ref{the 1}  the operator $-{\cal L}^{-1}$ is monotone in the space $BC(\R,\R)$.  
As proved in \cite{wuzou}, the operator $H$ enjoys similar properties:

\begin{lemma}
\label{lem 3.1} Assume that (\ref{A}). Then, for
any $\phi\in\Gamma,$ we have that
\begin{enumerate}
\item $H(\phi)(t)\ge 0,\ t\in\mathbb{R}$,
\item $H(\phi)(t)$ is nondecreasing in $t\in\mathbb{R,}$
\item $H(\psi)(t)\le H(\phi)(t)$ for all $t\in\mathbb{R}$, if $\psi\in
C(\mathbb{R},\mathbb{R})$ is given so that
$0\le \psi(t)\leq \phi(t)\le K$ for all
$t\in\mathbb{R}$.
\end{enumerate}
\end{lemma}
With this in mind, if we rewrite Eq.(\ref{wave}) in the form
\begin{equation}\label{fixed point}
\phi = -{\cal L}^{-1}H(\phi) , \ \phi \in \Gamma ,
\end{equation}
then, $\phi$ is a fixed point of the monotone operator $ -{\cal L}^{-1}H$ in $\Gamma$. Our next steps below will make sure that under some conditions $ -{\cal L}^{-1}H$ is well defined in $\Gamma$ or in a closed convex subset of $\Gamma$ where $-{\cal L}^{-1}H$ is a compact operator. Then, the fixed point of $-{\cal L}^{-1}H$ will be guaranteed by the Schauder Fixed Point Theorem.

\begin{definition} A function $\varphi \in BC^2(\R,\R)$ is called an upper solution (lower solution, respectively) for the wave equation (\ref{wave}) if it satisfies the following
\begin{align*}
& D\varphi''(t)-c\varphi'(t+r_1)+f^c(\varphi_{t+r_1})\le 0 , \\
& (D\varphi''(t)-c\varphi'(t+r_1)+f^c(\varphi_{t+r_1})\ge 0, \ \mbox{respectively})
\end{align*}
for all $t\in\R$.
\end{definition}

Below we will list some standing conditions on Eq.(\ref{wave}) before state further conditions for the existence of traveling waves to Eq.(\ref{RD1}):
\begin{enumerate}
\item[(H1)] $f(\hat 0) =f(\hat K)=0,$ where $\hat 0$, ($\hat K$, respectively) is the constant function $\phi (\theta) =0$ ($\phi (\theta )=K$, respectively), for all $\theta \in [-\tau_2,0]$;
\item[(H2)] There exists a positive constant $\beta$ such that
$$
f(\varphi)-f(\psi ) +\beta (\varphi(0)-\psi (0)) \ge 0
$$
for all $\varphi, \psi \in C([-\tau_2,0],\R)$ with $0\le \varphi (s)\le \phi (s)\le K$ for all $s\in [-\tau_2,0]$;
\item[(H3)] The operator $H$ is continuous in $BC(\R, [0,K])\to BC(\R,\R)$ and
$$
\sup_{\phi\in\Gamma} \| H(\phi)\| <\infty .
$$
\end{enumerate}
Before we proceed we set $F:=-{\cal L}^{-1}H.$
\begin{lemma}\label{lem 5.4}
Let $\phi \in BC^2(\R, [0,K])$. Assume further the standing assumptions (H1), (H2), (H3). Then,  $\phi$ is an upper solution (lower solution, respectively) of Eq.(\ref{wave}) if and only if
$$
F\phi \le \phi   \ (F\phi \ge \phi , \ \mbox{respectively}) .
$$
\end{lemma}
\begin{proof}
By definition, if $\phi$ is an upper solution, then
$$
 D\varphi''(t)-c\varphi'(t+r_1)+f^c(\varphi_{t+r_1})\le 0 .
$$
Since $\phi\in BC^2(\R,[0,K])$ and assumptions (H1), (H2) and (H3) we have
$$
{\cal L}\phi +H(\phi ) \le 0 .
$$
Therefore, as $-{\cal L}^{-1} $ is monotone, we have
$$
(-{\cal L}^{-1} )({\cal L}\phi  + H(\phi ) )\le 0.
$$
Consequently,
\begin{align}
-\phi -{\cal L}^{-1} H(\phi ) &\le 0.
\end{align}
Hence
$$
\phi \ge -{\cal L}^{-1} H(\phi )=F\phi .
$$
Conversely, we can reverse the argument to show that $\phi \ge F\phi$ implies $\phi$ is an upper solution. Similarly, we can prove the claim on the lower solutions.
\end{proof}

\begin{theorem}\label{the main}
Under the standing assumptions (H1), (H2), (H3), if there are an upper $\overline{\varphi} \in \Gamma$ and a lower solutions $\underline{\varphi}\in \Gamma $ of Eq.(\ref{wave}) such that for all $ \ t\in \R $
$$
0\le  \underline{\varphi}(t)\le \overline{\varphi}(t) .
$$
Then, there exists a monotone traveling wave solution $\phi$ to Eq.(\ref{wave}).
\end{theorem}
\begin{proof}
Set $\phi_n:=F^n(\overline{\varphi}), n\in\N$. As $H$ and ${-\cal L}^{-1}$ are both monotone, $F$ is monotone as well. By induction, we can show that $\phi_n$ is a monotone function for each $n$. In fact, denoting by $S^h$ the translation $\phi(\cdot ) \to \phi(\cdot -h)$, where we assume $h >0$ is a constant, in the function space $BC(\R,\R)$ we see that $S^h$ commutes with both $H$ and $-{\cal L}^{-1}$. Therefore, as $\phi \ge S^h\phi$, we have
$$
\phi_1 := F(\overline{\varphi} ) \ge F(S^h\overline{\phi})=S^h F(\overline{\varphi }),
$$
or
$$
\phi_1(t) \ge \phi_1 (t-h)
$$
for all $t\in \R$. Next, assume that $\phi_k(t)$ is a nondecreasing function. Then, using the above observation we can easily show by induction that $\phi_{k+1}(t)$ is also a nondecreasing function. This way, we obtain a sequence of bounded nondecreasing continuous functions $\{\phi_n(t)\}$. Moreover, 
$$
\underline{\varphi} \le \phi_n   \le \phi_{n+1}\le \overline{\varphi} , \ n\in \N .
$$
This shows that if there exists a subsequence of $\phi_n$ that is convergent in $\Gamma$, then the sequence itself must be convergent. 

\medskip
Consider the set 
$$
\Gamma_1:= \{ \phi\in \Gamma | \ \underline{\varphi} \le \phi \le \overline{\varphi}  \} .
$$
Clearly, that $\Gamma_1$ is a closed and convex subset of $BC(\R,\R)$. The restriction of $F$ to $\Gamma_1$ is well defined as an operator from $\Gamma_1$ to itself. In fact, let $\varphi \in \Gamma_1$, then as above we can show that $F(\varphi)$ is a monotone function on $\R$. Next, since
$
\underline{\varphi} \le \varphi \le \overline{\varphi},
$
and by assumption, we have
\begin{equation}\label{100}
\underline{\varphi} \le F(\underline{\varphi})  \le F(\varphi) \le F(\overline{\varphi}) \le \overline{\varphi}.
\end{equation}
Since
\begin{align*}
K & =\lim_{t\to\infty} \underline{\varphi}(t) \le  \lim_{t\to\infty}  F(\varphi) (t)\le \lim_{t\to\infty}  \overline{\varphi}(t) =K,\\
0 & =\lim_{t\to-\infty} \underline{\varphi} (t)\le  \lim_{t\to-\infty}  F(\varphi) (t)\le \lim_{t\to- \infty}  \overline{\varphi}(t) =0.
\end{align*}
by the Squeeze Theorem we  have 
\begin{align}
  \lim_{t\to\infty}  F(\varphi)(t)=K \
  \lim_{t\to-\infty}  F(\varphi) (t) =0.
\end{align}
This yields that
$ F(\varphi)  \in \Gamma$, so by (\ref{100}) $ F(\varphi)  \in \Gamma_1$.

Before proceeding we need the following lemma

\begin{lemma}\label{lem compact}
	For sufficiently small $r>0$, the operator $-{\cal L}^{-1}$ maps $\Gamma_1$ to a pre-compact subset of $BC(\R,\R)$.
\end{lemma}
\begin{proof}
	We will use the Arzela-Ascoli Theorem to prove the lemma. The boundedness of the set $-{\cal L}^{-1}(\Gamma_1)$ is clear from the exponential decay of $G(t,r)$ in $|t|$. In fact, since there exists positive constants $\delta, K_0$ such that $|G(t,r)| \le K_0e^{-\delta |t|}$ for all $t\in\R$
	\begin{align*}
		\|-{\cal L}^{-1} f \| & = \sup_{t\in\R}| -{\cal L}^{-1} f(t) | \\
		&= \sup_{t\in\R}
		| \int^\infty_{-\infty} G(t-s)f(s)ds| \\
		& \le\sup_{t\in\R} \int^\infty_{-\infty} |G(t-s,r)|\cdot |f(s)|ds \\
		&\le  \left( \int^\infty_{t} |G(t-s,r)|ds   + \int^t_{-\infty} |G(t-s)|ds     \right)\sup_{s\in\R}|f(s)| \\
		&\le \frac{2K_0}{\delta} \| f\| .
	\end{align*}
	Next, we are going to show that ${\cal L}^{-1}(\Gamma_1)$ is equicontinuous. Let $f\in \Gamma_1$ be any element. Then, since $f\in \Gamma_1$, $\| f\| \le \rho$, for a certain $\rho >0$ such that for any $x,y \in \R,$
	\begin{align*}
		|{\cal L}^{-1}f(x)-{\cal L}^{-1}f(y)| &=| \int^\infty_{-\infty}\left( G(x-s,r)-G(y-s,r)\right) f(s)ds| \\
		&\le    \int^\infty_{-\infty}\left| G(x-s,r)-G(y-s,r)\right| ds \cdot \sup_{s\in\R}| f(s)ds| \\
		&\le \rho \int^\infty_{-\infty}\left| G(x-s,r)-G(y-s,r)\right| ds.
	\end{align*}
	Given an $\epsilon_0 >0$, defined as $\epsilon:= \epsilon_0/(\rho K_0)$. $G(t,r)$ exponentially decays on each half interval of the real line, one can choose a sufficiently large number $N=N(\epsilon) >0$ so that  
	$$
	\int^\infty _N |G(s,r)|ds +\int^{-N}_{-\infty} |G(s,r)ds < \frac{ \epsilon}{8\rho }.
	$$
	Then, there exists a number $\delta_0>0$ dependant on $\epsilon$ such that 
	$$
	\left| G(x-s,r)-G(y-s,r)\right| < \frac{\epsilon}{8N\rho}
	$$
	for all $x,y \in [-N, -\epsilon/4]$ or $x,y\in [\epsilon/4, N]$. Finally, we have
	\begin{align*}
		|{\cal L}^{-1}f(x)-{\cal L}^{-1}f(y)| 
		&\le K \int^\infty_{-\infty}\left| G(x-s,r)-G(y-s,r)\right| ds\\
		&\le K\left( \int^\infty _N \left| G(x-s,r)-G(y-s,r)\right| ds +\int^{-N}_{-\infty} \left| G(x-s,r)-G(y-s,r)\right| ds\right) \\
		& \hspace{.5cm} +  K  \int^N_{\epsilon/4} \left| G(x-s,r)-G(y-s,r)\right| ds     +K  \int_{-N}^{-\epsilon/4}\left| G(x-s,r)-G(y-s,r)\right| ds  \\
		&\hspace{.5cm} + K  \int_{-\epsilon/4}^{\epsilon/4} \left| G(x-s,r)-G(y-s,r)\right| ds       \\
		&<  \frac{\epsilon}{4} +  \frac{\epsilon}{4} +\frac{\rho K_0\epsilon}{2}\\
		&= \epsilon_0,
	\end{align*}
	for all $x,y$ such that $|x-y|<\delta_0$ and $f\in \Gamma_1$. We are going to show that ${\cal L}^{-1}\Gamma$ is totally bounded. In fact, for each $\epsilon>0$ we can find a positive number $N$ such that 
	\begin{align*}
		| \underline \varphi (t) | < \epsilon , \ t < -N\\
		|\overline \varphi (t)-K| <\epsilon ,\ t > N .
	\end{align*}
	By the Arzela-Ascoli Theorem applying to the interval $[-N,N]$, the family of functions $\{f|_{[-N,N]}, f\in {\cal L}^{-1}\Gamma_1\}$ is totally bounded in the metric space $C([-N,N],\R)$. That yields that there are finitely many functions $\{f_1,f_2,\cdots , f_k\}$ functions of $ {\cal L}^{-1}\Gamma_1$ such that the balls $B_\epsilon (f_1),\cdots , B_\epsilon (f_k)$ is an open cover of $\{f|_{[-N,N]}, f\in {\cal L}^{-1}\Gamma_1\}$. By the way we choose $N$ we can conclude that the functions $B_\epsilon (f_1),\cdots , B_\epsilon (f_k)$ is an open cover of ${\cal L}^{-1}\Gamma_1\}$. That means ${\cal L}^{-1}\Gamma_1\}$ is precompact. The lemma is proved.
\end{proof}

\medskip
By Lemma \ref{lem compact}  the operator ${\cal L}^{-1}$ is compact, so the operator $F$ is compact as well. Therefore, $F$ is a continuous and compact operator from the closed and convex subset $\Gamma_1$ of the Banach space $BC(\R,\R)$. By the Schauder Fixed Point Theorem, $F$ must have a fixed point in $\Gamma_1$.
\end{proof}
When applying the results obtained in the above sections to particular models one is often faced with difficulty in constructing upper and lower solutions for the Monotone Iteration Method to work. To facilitate this process we can construct upper and lower solutions from rough functions, known as quasi-upper/lower solutions.

\begin{definition} A function $\varphi \in C^1(\R,\R),$ where $ \varphi, \varphi'$ are bounded on $\R$, $\varphi''$ is locally integrable and essentially bounded on $\R$ (that is, $\varphi''\in L^\infty$), is called a quasi- upper solution (quasi-lower solution, respectively) for the wave equation (\ref{wave}) if it satisfies the following for almost every $t\in \R$
\begin{align*}
& D\varphi''(t)-c\varphi'(t+r_1)+f^c(\varphi_{t+r_1})\le 0 , \\
& (D\varphi''(t)-c\varphi'(t+r_1)+f^c(\varphi_{t+r_1})\ge 0, \ \mbox{respectively}) .
\end{align*}
\end{definition}

\begin{definition}
A continuous function $\varphi:\r\to \R,$ that is  $C^1(\r,\r)$ a.e outside of the set $\T$  where  is called a super solution (sub solution, respectively) for the wave equation (\ref{wave}) if it satisfies the following for  $t\in \R/\T$
\begin{align*}
& \varphi''(t)-c\varphi'(t+r_1)+f_c(\varphi_{t+r_1})\le 0 , \\
& (\varphi''(t)-c\varphi'(t+r_1)+f_c(\varphi_{t+r_1})\ge 0, \ \mbox{respectively}) .
\end{align*}    
\end{definition}
\begin{proposition}\label{smoothprop1}
Let $\phi$ be a nondecreasing quasi-upper solution (quasi-lower solution, respectively) of Eq.(\ref{wave}) such that $\phi(t)\in [0,K]$ for all $t\in \R$. Then, $F\phi$ is a nondecreasing upper solution (lower solution, respectively) of Eq.(\ref{wave}).
\end{proposition}
\begin{proof}
The basic idea here is that via the operator ${\cal L}^{-1}$ the smoothness of the quasi-upper or lower solutions is improved. Therefore, $F\phi$ has enough smoothness to be an upper (or lower solution.
By Theorem \ref{the per2}, the operator  ${\cal L}$ (mapping $\varphi (\cdot )\mapsto \varphi''(\cdot )-c\varphi '(\cdot +r_1)-\beta u(\cdot +r_1)$) induces the operator $T: (\varphi , \varphi')^T\mapsto (0,f)$ that
is an isomorphism between $W^{1,\infty}$ and $L^\infty$. Moreover, if in addition, $f$ is continuous, then  
$\varphi$ is a classical solution, that is, $\varphi$ is twice continuously differentiable and $\varphi,\varphi',\varphi''$ are bounded. This yields that $F\phi =-{\cal L}^{-1}H\phi$ is of class $C^2$ since $H\phi$ is continuous and bounded.
Since $\phi$ is a quasi-upper solution, arguing in the same manner as in Lemma \ref{lem 5.4} we can show that
$$
\phi \ge F\phi .
$$
Consequently, $F\phi \in BC^2(\R,[0,K])$. As $F$ is monotone, this yields $\phi \ge F\phi \ge F(F\phi)$. In particular, this shows that $F\phi$ satisfies all conditions of Lemma \ref{lem 5.4} to be an upper solution of Eq.(\ref{wave}).
\end{proof}

\medskip
\noindent In order to simplify the construction of quasi upper/lower solutions we can define the Fourier transform \[\mathcal{F}\left(\phi\right)(t)=\int^\infty_{-\infty} G(t-s,r_1)H(\phi(s))ds,,\] 
where $H(\phi(s))$ is continuous on $\r$ and continuously differentiable a.e, then we have the following.
\begin{lemma}
Let ${\phi}$ be a super solution (sub solution, respectively) of Eq. (\ref{wave}). Then, $\mathcal{F}(\phi)$ is an quasi-upper solution (quasi-lower solution, respectively) of Eq. 
(\ref{wave}).
\end{lemma}
\begin{proof} This can be shown in a similar manner to Propostion \ref{smoothprop1}.
\end{proof}
\begin{corollary}
Assume $(H1)$ and $ (H2)$  hold, if there is an super solution $\overline{\phi}$ in $\Gamma$ and a sub solution $\underline{\phi}$ for Eq.(\ref{wave}) not necessarily  in $\Gamma$  such that for all $t\in \R$
\[0\le  \underline{\phi}(t)\le \overline{\phi}(t).\]
Then, there exists a monotone traveling wave solution to the equation (\ref{wave}).
\end{corollary}

\begin{lemma}\label{lem 3.7}
Assume that $\phi(t)$ is a differentiable function such that $\phi'(t)$ is uniformly continuous
and the limit
\begin{equation}\label{lim}
\lim_{t\to+\infty} \phi(t) = a. 
\end{equation}
Then, $\lim_{t\to+\infty} \phi'(t)=0$.
\end{lemma}
\begin{proof}
Assuming to the contrary that $\lim_{t\to+\infty}\phi'(t)\not=0$. Then, there exists a sequence $\{t_n\}\to\infty$ such that $\inf_{n\in\N}|\phi'(t_n)| >\epsilon$.
As $\phi'(\cdot )$ is uniformly continuous, for each positive constant $\epsilon>0$, there exists a positive $\delta$ such that if $|t-s|<\delta$, then, 
\begin{equation}
|\phi'(t)-\phi'(s)| \le \frac{\epsilon}{2}.
\end{equation}
By (\ref{lim}), for each positive $\epsilon$ there exists a large constant $N$ such that 
\begin{equation}\label{3.12}
|\phi(t)-\phi(s)| \le \frac{\delta \epsilon}{4}
\end{equation}
for all $t,s\ge N$. 
On the other hand, we have
\begin{align*}
|\phi(t_n+\delta/2)-\phi(t_n-\delta/2) | &= |\int^{t_n+\delta/2}_{t_n-\delta/2} \phi'(s)ds | \\
& = |\int^{t_n+\delta/2}_{t_n-\delta/2} \phi'(t_n)ds + \int^{t_n+\delta/2}_{t_n-\delta/2} (\phi'(s)-\phi'(t_n))ds | \\
&=|\delta \phi'(t_n) +  \int^{t_n+\delta/2}_{t_n-\delta/2} (\phi'(s)-\phi'(t_n))ds | \\
&=|\delta \phi'(t_n) |-| \int^{t_n+\delta/2}_{t_n-\delta/2} (\phi'(s)-\phi'(t_n))ds | \\
&\ge \delta\epsilon - \delta \epsilon  /2                        \\
 & =\delta \epsilon  /2   .
\end{align*}
This contradicts (\ref{3.12}). That is $\lim_{t\to+\infty}\phi'(t)=0$.
\end{proof}
\begin{corollary}
Assume that $\phi$ is a solution of Eq.(\ref{wave}) such that
\begin{equation}
\sup_{t >0} |\phi''(t)| <\infty 
\end{equation}
and the limit
$$
\lim_{t\to+\infty} \phi(t) = a. 
$$
Then, $f(\hat a)=0$, where $\hat a$ is the constant function $\varphi(\theta)=a$ for all $\theta \in [-\tau_2,0]$.
\end{corollary}
\begin{proof}
First, by Lemma \ref{lem 3.7}, $\lim_{t\to+\infty} \phi'(t)=0$. Therefore, the function $\varphi:=\phi'$ satisfies all conditions of Lemma \ref{lem 3.7}, so
$\lim_{t\to+\infty} \varphi'(t)=\lim_{t\to+\infty} \phi''(t)=0$. Subsequently
$$
\lim_{t\to+\infty} f^c(\phi_t)=0.
$$
Since $f$ is Lipschitz continuous, 
$$
f^c( \lim_{t\to+\infty} \phi_t)=f^c(\hat a)=0.
$$
Finally, this yields that $f(\hat a)=0.$ 
\end{proof}

The following improves Theorem \ref{the main}.
\begin{theorem}\label{the main2}
Under the standing assumptions (H1), (H2), (H3), if there is an upper $\overline{\varphi} \in \Gamma$ and a lower solutions $\underline{\varphi}$
that is not necessarily in $\Gamma$ of Eq.(\ref{wave}) such that for all $ \ t\in \R $
$$
0\le  \underline{\varphi}(t)\le \overline{\varphi}(t) 
$$
and 
$$
\lim_{t\to+\infty} \underline{\varphi}(t)=a \not=0.
$$
Then, there exists a monotone traveling wave solution $\phi$ to Eq.(\ref{wave}).
\end{theorem}
\begin{proof}
We follow the proof of Theorem \ref{the main} by considering the operator $F$ in $\Gamma$. Since the set $\{ \phi_n:=F^n\overline{\varphi} \}$ is precompact, it has a convergent subsequence. Due to the monotone property of this sequence $\phi_n$, the sequence itself is convergent, say, to $\phi_0$. Obviously, $\phi_0$ is nondecreasing and $\lim_{t\to-\infty}\phi_0(t)=0$ and the limit $\lim_{t\to+\infty}\phi_0(t)$ exists, say, equals $a\in [0,K]$. As $\phi_0(t)\ge \underline{\varphi} (t)$, $a>0$. As there is no equilibrium between $0$ and $K$, this follows that $a=K$. Hence, $\phi_0$ is a monotone wave solution of Eq.(\ref{wave}).
\end{proof}
\begin{remark}
The theory we presented above is for one dimensional systems. However, it can be easily extended to multi-dimensional systems where the diffusion is diagonal, that is, the multi-dimensional systems are of the form
\begin{equation}
\begin{cases}
\frac{\partial u_1(x,t)}{\partial t}=D_1\frac{\partial^2 u_1(x,t-\tau_1)}{\partial x^2} +f_1(u_t)\\
\frac{\partial u_2(x,t)}{\partial t}=D_2\frac{\partial^2 u_2(x,t-\tau_2)}{\partial x^2} +f_2(u_t)\\
\cdots  \\
\frac{\partial u_n(x,t)}{\partial t}=D_n\frac{\partial^2 u_n(x,t-\tau_n)}{\partial x^2} +f_n(u_t),
\end{cases}
\end{equation}
where $\tau_i, i=i,2,\cdots ,n$ are sufficiently small positive constants, $D_i,i=1,2,\cdots,n$ are positive constants, $u=(u_1,\cdots ,u_n)^T$, $u_t(\theta ):= u(t+\theta )$, $\theta \in [-\tau, 0]$ with given positive $\tau$.
\end{remark}

\section{Applications}\label{applications}

\subsection{Belousov-Zhabotinskii equations}
In this section we consider the
existence of traveling waves to Belousov-Zhabotinskii Equations with delay in both diffusion and reaction terms
\begin{equation}\label{pde-bz}
\begin{cases}
\frac{\partial }{\partial t}u(x,t) = \frac{\partial^2}{\partial
x^2} u(x,t-\tau_1) +u(x,t) [ 1-u(x,t) -r v(x,t-\tau_2)] ;\\
\frac{\partial }{\partial t}v(x,t) = \frac{\partial^2}{\partial x^2}
v(x,t) -bu(x,t)v(x,t) ,
\end{cases}
\end{equation}
where $r,b,  \tau_1,  \tau_2  $ are positive constants, $u$ and $v$ are scalar functions.
As shown in \cite{wuzou}, the function $f(\phi):= (f_1(\phi),f_2(\phi))^T$ defined as
\begin{align}
f_1(\phi)& := \psi_1(0)[s-\phi_1(0)+r\phi_2(-\tau_2)]\\
f_2(\phi)& := b\phi_1(0)[1-\phi(0)],
\end{align}
where $s:=1-r$, $\phi \in C[-c\tau_2,\R^2)$ satisfies
$$
f_c(\phi )-f_c(\psi )+\beta [\phi(0)-\psi(0)] \ge 0,
$$
where $\beta =diag(\beta_1,\beta_2)$ with $\beta_1\ge 2-s$ and $\beta_2\ge b$. Therefore, using the theory in the previous section we can find quasi-upper and quasi-lower solutions of the Belousov-Zhabotinskii model. 

\medskip
The associated wave equation is of the form
\begin{equation}\label{waveBZ}
\begin{cases}
\varphi_{1}^{\prime\prime}(t)-c\varphi_{1}^{\prime}(t+r_1)+\varphi_{1}(t+r_1)\left(
(1-r)-\varphi_{1}(t+r_1)-r\varphi_{2}(t+r_1-r_2)\right)  =0\\
\varphi_{2}^{\prime\prime}(t)-c\varphi_{2}^{\prime}(t+r_1)+b\varphi_{1}(t+r_1)\left(
1-\varphi_{2}(t+r_1)\right)  =0.
\end{cases},
\end{equation}
where $r_1:=c\tau_1, r_2:=c\tau_2$.

\subsubsection{Quasi-upper solutions}
Define the numbers $\lambda_{0}$ and $\mu_{0}$ as
\[
\lambda_{0}=\frac{c+\sqrt{c^{2}-4}}{2},\;\;\;\mu_{0}=\frac{c+\sqrt{c^{2}-4b}
}{2},
\]
which are the  roots of the characteristic equations
\begin{align}
\lambda^{2}-c\lambda+1&=0, \label{40}\\
\mu^{2}-c\mu+b &=0 \label{41},
\end{align}
respectively. Observe that since
$
1< b
$
then,
$\lambda_{0} > \mu_{0}.
$

\begin{claim}\label{claim 4}
Let $c>2$ and $U$ be an open strip $\{ z\in \C | \lambda_0-\epsilon < \Re z < \lambda_0+\epsilon\}$ so that it does not include the other root of (\ref{40}) in it. Then, for sufficiently small $r_1$ there exists only a single root $\lambda_1(r_1)$ of the equation
\begin{equation}\label{l1}
\lambda^2-c\lambda e^{r_1\lambda}+e^{r_1\lambda }=0.
\end{equation}
in $U$ that depends continuously on $r_1$. Moreover,  $\lambda_1(r_1)$ is real and
\begin{equation}
\lim_{s\to 0} \lambda_1(r_1) =\lambda_0 .
\end{equation}
\end{claim}
\begin{proof}
The proof can be done in the same manner as in the proof of Part (iii) of Proposition \ref{pro 1} using the Rouch\'e Theorem. Due to the uniqueness of the root we see that $\lambda_1$ must be the same as $\overline{\lambda_1}$. That is, $\lambda_1$ is real.
\end{proof}
\begin{claim}\label{claim 5}
Let $1<b, 2\sqrt{b}<c$ and $V$ be an open strip $\{ z\in \C | \mu_0-\epsilon < \Re z < \mu_0+\epsilon\}$ so that it does not include the other root of (\ref{41}) in it. Then, for sufficiently small $r_1$ there exists only a single root $\mu_1(r_1)$ of the equation
\begin{equation}\label{m1}
\mu^2-c\mu e^{r_1\mu}+be^{r_1\mu }=0.
\end{equation}
in $V$ that depends continuously on $r_1$. Moreover, $\mu_1(r_1)$ is real and
\begin{equation}
\lim_{s\to 0} \mu_1(r_1) =\mu_0 .
\end{equation}
\end{claim}
\begin{proof}
The proof can be done in the same manner as in the proof of Part (iii) of Proposition \ref{pro 1} using the Rouch\'e Theorem. Due to the uniqueness of the root we see that $\mu_1$ must be the same as $\overline{\mu_1}$. That is, $\mu_1$ is real.
\end{proof}
Note that as $\lambda_{0}> \mu_{0}$ for sufficiently small $r_1$ we have $0<\mu_1 <\lambda_1$. Let us define functions $\varphi_{1}$ and $\varphi_{2}$ as follows:
\[
\varphi_{1}(t):=\left\{
\begin{array}
[c]{l}%
\frac{1}{2}e^{\lambda_{1}t},\;\;\;\;\;\;\quad t\leq0,\\
1-\frac{1}{2}e^{-\lambda_{1}t},\quad t>0
\end{array}
\right.  \;\;\;\;\varphi_{2}(t):=\left\{
\begin{array}
[c]{l}%
\frac{1}{2}e^{\mu_{1}t},\;\;\;\;\;\quad t\leq0,\\
1-\frac{1}{2}e^{-\mu_{1}t},\quad t>0
\end{array}
\right.
\]
observe that  for sufficiently small $r_1$, $0<\varphi_{1}(t)<1\;$and similarly for
$0<\varphi_{2}(t)<1.$

First, it is easily seen that
\begin{align*}
\varphi_{1}^{\prime}(t) &  =\left\{
\begin{array}
[c]{l}%
\frac{\lambda_{1}}{2}e^{\lambda_{1}},\quad t\leq0,\\
\frac{\lambda_{1}}{2}e^{-\lambda_{1}t},\quad t>0
\end{array}
\right.  ,\quad\varphi_{1}^{\prime\prime}(t)=\left\{
\begin{array}
[c]{l}%
\frac{\lambda_{1}^{2}}{2}e^{\lambda_{1}t},\quad t\leq0,\\
\frac{-\lambda_{1}^{2}}{2}e^{-\lambda_{1}t},\quad t>0
\end{array}
\right.  \\
\varphi_{2}^{\prime}(t) &  =\left\{
\begin{array}
[c]{l}%
\frac{\mu_{1}}{2}e^{\mu_{1}t},\quad t\leq0,\\
\frac{\mu_{1}}{2}e^{-\mu_{1}t},\quad t>0
\end{array}
\right.  ,\quad\varphi_{2}^{\prime\prime}(t)=\left\{
\begin{array}
[c]{l}%
\frac{\mu_{1}^{2}}{2}e^{\mu_{1}t},\quad t\leq0,\\
\frac{-\mu_{1}^{2}}{2}e^{-\mu_{1}t},\quad t>0 .
\end{array}
\right.
\end{align*}
Note that $\varphi_1', \varphi_2'$ are both continuous and bounded on $\R$ and $\varphi_1'', \varphi_2''$ exist and are continuous everywhere and bounded except for $t=0$.

\begin{claim}\label{claim 6}
For sufficiently small $r_1$ and $c>2\sqrt{b}$, the vector function $(\varphi_1 ,\varphi_2)^T$ is a quasi-upper solution of (\ref{waveBZ}).
\end{claim}
\begin{proof}
{\bf Case $t\le - r_1 $}: 
Plugging these functions into the first equation in (\ref{waveBZ}) we have
for $t+r_1 \leq0$
\begin{align*}
&  \left[
\varphi_{1}^{\prime\prime}(t)-c\varphi_{1}^{\prime}(t+r_1)+\varphi
_{1}(t+r_1)\right]  -r\varphi_{1}(t+r_1)(1-\varphi_{2}(t+r_1-r_2))-\varphi_{1}^{2}(t+r_1)\\
&  =\left[  \frac{\lambda_{1}^{2}}{2}e^{\lambda_{1}t}-c\frac{\lambda_{1}} 
{2}e^{\lambda_{1}(t+r_1)}+\frac{1}{2}e^{\lambda_{1}(t+r_1)}\right]  -re^{\lambda_{1} 
(t+r_1)}\left(  1-\varphi_{2}(t+r_1-r_2)\right)  -\frac{1}{4}e^{2\lambda_{1}(t+r_1)}\\
&  =\left[  \lambda_{1}^{2}-c\lambda_{1}e^{r_1\lambda_1}+e^{r_1\lambda_1}\right]  \frac{1}{2}e^{\lambda_{1} 
t}-re^{\lambda_{1}(t+r_1)}\left(  1-\frac{1}{2} e^{\mu_1(t+r_1-r_2)}\right)  -\frac{1} 
{4}e^{2\lambda_{1}(t+r_1)}\\
&  =-re^{\lambda_{1}(t+r_1)}\left(  1-\frac{1}{2} e^{\mu_1(t+r_1-r_2)}\right)  -\frac{1} 
{4}e^{2\lambda_{1}(t+r_1)}
\leq 0.
\end{align*}

\medskip
For the second equation in (\ref{waveBZ}), for $t\le -r_1$ we have 
\begin{align*}
&  \varphi_{2}^{\prime\prime}(t)-c\varphi_{2}^{\prime}(t+r_1)+b\varphi
_{1}(t+r_1)\left(  1-\varphi_{2}(t+r_1)\right)  \\
&
=\frac{\mu_{1}^{2}}{2}e^{\mu_{1}t}-c\frac{\mu_{1}}{2}e^{\mu_{1}(t+r_1)}+\frac
{b}{2}e^{\lambda_{1}(t+r_1)}\left(  1-\frac{1}{2}e^{\mu_{1}(t+r_1)}\right)  \\
&  =\frac{1}{2}\left( \mu^2_1-c\mu_1e^{r_1\mu_1}+be^{r_1\mu_1} \right)  e^{\mu_1t} 
+\frac{b}{2} \left(  e^{\lambda_1(t+r_1)} -e^{\mu_1(t+r_1)}     \right)
-\frac{b}{4} e^{(\lambda_1+\mu_1)(t+r_1)}\\
& = \frac{b}{2} \left(  e^{\lambda_1(t+r_1)} -e^{\mu_1(t+r_1)}     \right)
-\frac{b}{4} e^{(\lambda_1+\mu_1)(t+r_1)} \le 0
\end{align*}
since $\lambda_1 >\mu_1$ and $t+r_1 \le 0$.

\medskip\noindent
{\bf Case: $-r_1 \le t\le 0$}:
\begin{align*}
A:=&  \left[
\varphi_{1}^{\prime\prime}(t)-c\varphi_{1}^{\prime}(t+r_1)+\varphi
_{1}(t+r_1)\right]  -r\varphi_{1}(t+r_1)(1-\varphi_{2}(t+r_1-r_2))-\varphi_{1}^{2}(t+r_1)\\
&  =\left[  \frac{\lambda_{1}^{2}}{2}e^{\lambda_{1}t}-c\frac{\lambda_{1}} 
{2}e^{-\lambda_{1}(t+r_1)}+1-\frac{1}{2}e^{-\lambda_{1}(t+r_1)}\right]  \\
& \hspace{.5cm} -r\left( 1-\frac{1}{2}e^{-\lambda_{1}(t+r_1)}\right) \left(  1-\varphi_{2}(t+r_1-r_2)\right)  -\varphi_{1}^{2}(t+r_1)\\
\end{align*}
As $\lambda_1$ is a root of Eq.(\ref{l1})
\begin{align*}
A&= \frac{c\lambda_1}{2}\left(e^{\lambda_1(t+r_1)}-e^{-\lambda_1(t+r_1)} \right) + 1-\frac{1}{2}\left( e^{\lambda_1(t+r_1)}+e^{-\lambda_1(t+r_1)} \right) \\
& \hspace{.5cm} -r\left( 1-\frac{1}{2}e^{-\lambda_{1}(t+r_1)}\right) \left(  1-\varphi_{2}(t+r_1-r_2)\right)  -\varphi_{1}^{2}(t+r_1)\\
&  =c\lambda_1 \sinh (\lambda_1(t+r_1))+1-\cosh (\lambda_1(t+r_1)) \\
& \hspace{.5cm} -r\left( 1-\frac{1}{2}e^{-\lambda_{1}(t+r_1)}\right) \left(  1-\varphi_{2}(t+r_1-r_2)\right)  -\varphi_{1}^{2}(t+r_1).
\end{align*}
As $-r_1\le t\le 0$ and $r_1$ is sufficiently small, using the power expansions of $\sinh x$ and $\cosh x$ we have
$$
\sinh (\lambda_1(t+r_1)) \approx \lambda_1(t+r_1 )+o(r^2_1), \quad 1-\cosh (\lambda_1(t+r_1))\approx - \frac{(\lambda_1(t+r_1))^2}{2}+o(r^2_1).
$$
Finally, we have 
\begin{align*}
A&  = c\lambda_1^2(t+r_1) - \frac{(\lambda_1(t+r_1))^2}{2}+o(r^2_1) \\
&\hspace{.5cm }  -r\left( 1-\frac{1}{2}e^{-\lambda_{1}(t+r_1)}\right) \left(  1-\varphi_{2}(t+r_1-r_2)\right)  -\varphi_{1}^{2}(t+r_1)\le 0,
\end{align*}
because when $r_1$ is sufficiently small, the limit of the right hand side is $-(r+1)/4$ as $r_1\to 0$.

\medskip
For the second equation we have
\begin{align*}
B&:= \varphi_{2}^{\prime\prime}(t)-c\varphi_{2}^{\prime}(t+r_1)+b\varphi_{1}(t+r_1)\left(
1-\varphi_{2}(t+r_1)\right)  \\
&= \frac{\mu_1^2}{2}e^{\mu_1t}-\frac{c\mu_1}{2}e^{-\mu_1(t+r_1)}+b\left( 1-\frac{1}{2}e^{-\lambda_1(t+r_1)}\right) \frac{1}{2}e^{-\mu_1(t+r_1)}\\
&= \left[ \frac{\mu_1^2}{2}e^{\mu_1t} -\frac{c\mu_1}{2}e^{\mu_1(t+r_1)} +\frac{b}{2}e^{\mu_1(t+r_1)} \right] 
+\frac{c\mu_1}{2}e^{\mu_1(t+r_1)}  -\frac{c\mu_1}{2}e^{-\mu_1(t+r_1)} \\
&\hspace{0.5cm} - \frac{b}{2}e^{\mu_1(t+r_1)}  +\frac{b}{2}e^{-\mu_1(t+r_1)} -\frac{b}{4}e^{-(\lambda_1+\mu_1)(t+r_1)} .
\end{align*}
As $\mu_1$ is a root of Eq.(\ref{m1}), 
\begin{align*}
B&=
\frac{c\mu_1}{2}e^{\mu_1(t+r_1)}  -\frac{c\mu_1}{2}e^{-\mu_1(t+r_1)}  - \frac{b}{2}e^{\mu_1(t+r_1)}  +\frac{b}{2}e^{-\mu_1(t+r_1)} -\frac{b}{4}e^{-(\lambda_1+\mu_1)(t+r_1)} .
\end{align*}
As $-r_1\le t\le 0$ and the limit of the right hand side is $-b/4<0$ as $r_1\to 0$, it follows that for sufficiently small $r_1$ we will have  $B\le 0$.

\medskip
{\bf Case  $t>0$}:
\begin{align*}
A&:= \left[
\varphi_{1}^{\prime\prime}(t)-c\varphi_{1}^{\prime}(t+r_1)+\varphi
_{1}(t+r_1)\right]  -r\varphi_{1}(t+r_1)(1-\varphi_{2}(t+r_1-r_2))-\varphi_{1}^{2}(t+r_1)\\
&  = \left[  \frac{-\lambda_{1}^{2}}{2}e^{-\lambda_{1}t}-c\frac{\lambda_{1}} 
{2}e^{-\lambda_{1}(t+r_1)}+1-\frac{1}{2}e^{-\lambda_{1}(t+r_1)}\right]  \\
& \hspace{.5cm}-r\varphi_{1}(t+r_1)(1-\varphi_{2}(t+r_1-r_2))-\varphi_{1}^{2}(t+r_1)\\
 &=  \left[  \frac{-\lambda_{1}^{2}}{2}e^{-\lambda_{1}t}+c\frac{\lambda_{1}} 
{2}e^{-\lambda_{1}t+\lambda_1 r_1}-\frac{1}{2}e^{-\lambda_{1}t+\lambda_1r_1}\right] \\
& -c\frac{\lambda_{1}} 
{2}e^{-\lambda_{1}t+\lambda_1 r_1} -c\frac{\lambda_{1}} 
{2}e^{-\lambda_{1}t-\lambda_1 r_1} +\frac{1}{2}e^{-\lambda_{1}t+\lambda_1 r_1} +1-\frac{1}{2}e^{-\lambda_{1}(t+r_1)} \\
&-r\varphi_{1}(t+r_1)(1-\varphi_{2}(t+r_1-r_2))-\varphi_{1}^{2}(t+r_1) 
 \end{align*}
Since $\lambda_1$ is a root of Eq.(\ref{l1}), we have
\begin{align*}
&  \frac{\lambda_{1}^{2}}{2}e^{-\lambda_{1}t}-c\frac{\lambda_{1}} 
{2}e^{-\lambda_{1}t+\lambda_1 r_1}+\frac{1}{2}e^{-\lambda_{1}t+\lambda_1r_1}=0.
\end{align*}
Hence,
\begin{align*}
A&= -c\frac{\lambda_{1}} 
{2}e^{-\lambda_{1}t+\lambda_1 r_1} -c\frac{\lambda_{1}} 
{2}e^{-\lambda_{1}t-\lambda_1 r_1} +\frac{1}{2}e^{-\lambda_{1}t+\lambda_1 r_1} +1-\frac{1}{2}e^{-\lambda_{1}(t+r_1)} \\
&-r\varphi_{1}(t+r_1)(1-\varphi_{2}(t+r_1-r_2))-\varphi_{1}^{2}(t+r_1) \\
&= \left[ 1-c\lambda_1 \cosh(\lambda_1r_1)+\sinh(\lambda_1r_1)\right] e^{-\lambda_1 t} \\
& -r\varphi_{1}(t+r_1)(1-\varphi_{2}(t+r_1-r_2))-\varphi_{1}^{2}(t+r_1).
\end{align*}
Since
$$
\lim_{r_1\downarrow 0} \left[ 1-c\lambda_1 \cosh(\lambda_1r_1)+\sinh(\lambda_1r_1)\right] <0,
$$
it follows that $A\le 0$ for sufficiently small $r_1$.

For the second equation
\begin{align*}
B&:= \varphi_{2}^{\prime\prime}(t)-c\varphi_{2}^{\prime}(t+r_1)+b\varphi_{1}(t+r_1)\left(
1-\varphi_{2}(t+r_1)\right)  \\
&= -\frac{\mu_1^2}{2}e^{-\mu_1t}-\frac{c\mu_1}{2}e^{-\mu_1(t+r_1)}+b\left( 1-\frac{1}{2}e^{-\lambda_1(t+r_1)}\right) \frac{1}{2}e^{-\mu_1(t+r_1)}\\
&= \left[  -\frac{\mu_1^2}{2}e^{-\mu_1t} +\frac{\mu_1}{2}e^{-\mu_1t+\mu r_1} -\frac{b}{2}e^{-\mu_1 t+\mu_1r_1)} \right] 
\\
& \hspace{0.5cm} -\frac{c\mu_1}{2}e^{-\mu_1t+\mu r_1} -\frac{c\mu_1}{2}e^{-\mu_1t-\mu r_1} +\frac{b}{2}e^{-\mu_1 t+\mu_1r_1)} \\
& \hspace{0.5cm}+ \frac{b}{2}e^{-\mu_1(t+r_1)} -\frac{b}{4}e^{-(\lambda_1+\mu_1)(t+r_1)} .
\end{align*}
As $\mu_1$ is a root of Eq.(\ref{m1}),
\begin{align*}
B&= -\frac{c\mu_1}{2}e^{-\mu_1t+\mu_1 r_1} -\frac{c\mu_1}{2}e^{-\mu_1t-\mu_1 r_1} +\frac{b}{2}e^{-\mu_1 t+\mu_1r_1} \\
& \hspace{0.5cm} +\frac{b}{2}e^{-\mu_1(t+r_1)} -\frac{b}{4}e^{-(\lambda_1+\mu_1)(t+r_1)} \\
& = \left(\frac{b}{2}-c\mu_1 \right)\cosh (\mu_1 r_1)e^{-\mu_1t} -\frac{b}{4}e^{-(\lambda_1+\mu_1)(t+r_1)} .
\end{align*}
Since $c>2\sqrt{b}$ we have $c^2>4b$, and 
$$
\frac{b}{2} < \frac{c^2}{8} < \frac{c(c+\sqrt{c^2-4b})}{2}=c\mu_1 .
$$
This yields  $B\le 0.$ The claim is proved.
\end{proof}

\subsubsection{Quasi-lower solutions}
To construct a quasi-lower solution we consider the equation
\begin{equation}\label{l2}
\lambda^2 -c\lambda e^{r_1 \lambda}+\frac{e^{r_1\lambda}}{2}=0 .
\end{equation}
Arguing as in Claim \ref{claim 4}, for sufficiently small $r_1$ Eq.(\ref{l2}) has exactly two real roots
 in the small neighborhood of two distinct roots 
 $$
\eta_{1,2} =\frac{c\pm \sqrt{c^2-2}}{2} 
$$
of the equation
$$
\lambda ^2 -c\lambda +\frac{1}{2}=0 .
$$ 
We will denote by $\lambda_2$ the real solution of Eq(\ref{l2}) in the small neighborhood of $\eta_2=(c+\sqrt{c^2-2})/2$. Note that for sufficiently small $r_1$, since $\eta_2 >  \lambda_0$, we have $ \lambda_2>\lambda_1$. Next, define 
$$
\underline{\varphi_2}(t)=0, t\in \R ,\quad 
\underline{ \varphi_1}(t)=\begin{cases} \frac{ e^{\lambda_2t}}{2k}, \ t\le 0 ,\\
\frac{1}{2k} , \ t> 0.
\end{cases}
$$

\medskip
\begin{claim}
Assume that $0<r\le 1/4$ and $k\ge2$ is sufficiently large number. Then, $\underline{\varphi}(t):=(\underline{\varphi}_1(t),\underline{\varphi}_2 (t))^T$ is a sub solution of Eq.(\ref{waveBZ}) and 
$$
0\le \underline{\varphi} (t) \le \overline{\varphi}(t) \le 1, \ t\in\R ,
$$
where $\overline{\varphi}=(\varphi_1,\varphi_2)^T$ is defined from Claim \ref{claim 6}.
\end{claim}
\begin{proof}
It is easy to see that $0\le \underline{\varphi} (t) \le \overline{\varphi}(t) $. Next, we will verify that $\underline{\varphi} (t)$ is a quasi-lower solution of Eq.(\ref{waveBZ}). Substituting the function $ \underline{\varphi} (t) $ into the second equation of Eq.(\ref{waveBZ}) yields that \begin{align}
\underline{\varphi}_{2}^{\prime\prime}(t)-c\underline{\varphi}_{2}^{\prime}(t+r_1)+b\underline{\varphi}_{1}(t+r_1)\left(
1-\underline{\varphi}_{2}(t+r_1)\right)  = b\underline{\varphi}_{1}(t+r_1)\ge 0.
\end{align}
\noindent For the first equation we will look at three cases.

\medskip \noindent 
\noindent{\bf Case 1: $t< - r_1 $}: We have
\begin{align*}
&\underline{\varphi}_{1}^{\prime\prime}(t)-c\underline{\varphi}_{1}^{\prime}(t+r_1)+\underline{\varphi}_{1}(t+r_1)\left(
(1-r)-\underline{\varphi}_{1}(t+r_1)-r\underline{\varphi}_{2}(t+r_1-r_2)\right) \\
&= \underline{\varphi}_{1}^{\prime\prime}(t)-c\underline{\varphi}_{1}^{\prime}(t+r_1)+\underline{\varphi}_{1}(t+r_1)\left(
(1-r)-\underline{\varphi}_{1}(t+r_1)\right) \\
&= \left[ \underline{\varphi}_{1}^{\prime\prime}(t)-c\underline{\varphi}_{1}^{\prime}(t+r_1)+\frac{1}{2}  \underline{\varphi}_{1}(t+r_1)\right]
+ \underline{\varphi}_{1}(t+r_1)\left(
\frac{1}{2}-r-\underline{\varphi}_{1}(t+r_1)\right) \\
&=\underline{\varphi}_{1}(t+r_1)\left(
\frac{1}{2}-r-\underline{\varphi}_{1}(t+r_1)\right) \ge 0,
\end{align*}
since, $\underline{\varphi}_{1}^{\prime\prime}(t)-c\underline{\varphi}_{1}^{\prime}(t+r_1)+\frac{1}{2}  \underline{\varphi}_{1}(t+r_1)=0,$  $0<r\le 1/4$ and $\underline{\varphi}_{1}(t+r_1)\le 1/4.$

\medskip \noindent {\bf Case 2: $- r_1\le t< 0 $}. It is important to note that there are know subcases here since $\underline{\varphi}_{2}(t+r_1-r_2)=0.$

\begin{align*}
&B=\underline{\varphi}_{1}^{\prime\prime}(t)-c\underline{\varphi}_{1}^{\prime}(t+r_1)+\underline{\varphi}_{1}(t+r_1)\left(
(1-r)-\underline{\varphi}_{1}(t+r_1)-r\underline{\varphi}_{2}(t+r_1-r_2)\right) \\
&=\underline{\varphi}_{1}^{\prime\prime}(t)+\underline{\varphi}_{1}(t+r_1)\left(
(1-r)-\underline{\varphi}_{1}(t+r_1)\right),
\end{align*}
since $\underline{\varphi}_{1}(t+r_1)$ is constant. Furthermore, $\underline{\varphi}_{1}^{\prime\prime}(t)>0, (1-r)-\underline{\varphi}_{1}(t+r_1)>0,$ so $B> 0.$   
\end{proof}

\medskip \noindent {\bf Case 2: $ 0\le t $}.
\begin{align*}
&B=\underline{\varphi}_{1}^{\prime\prime}(t)-c\underline{\varphi}_{1}^{\prime}(t+r_1)+\underline{\varphi}_{1}(t+r_1)\left(
(1-r)-\underline{\varphi}_{1}(t+r_1)-r\underline{\varphi}_{2}(t+r_1-r_2)\right) \\
&=\underline{\varphi}_{1}(t+r_1)\left(
(1-r)-\underline{\varphi}_{1}(t+r_1)\right)>\ge 0,
\end{align*}
since $\underline{\varphi}_{1}(t)$ is constant. 
\begin{corollary}
Assume that $1<k, 0<r\le 1/4$, $2<2\sqrt{b} < c$. Then, Eq.(\ref{waveBZ}) has a traveling wave solution if the delays $\tau_1,\tau_2$ are sufficiently small.
\end{corollary}

\subsection{Traveling waves for Fisher-KPP equations with delay in diffusion}
We will consider the traveling wave problem for equation
\begin{equation}\label{Classmodel}
\frac{\partial u(x,t)}{\partial t}= D\frac{\partial u(x,t-\tau_1)}{\partial x^2}+ u(x,t-\tau_2)\left(1-u(x,t)\right), \ \tau_1, \tau_2>0.   
\end{equation}
In fact, if we take the make the substitution $\phi(x+ct)=u(x,t)$ as earlier, then for $\xi=x+ct, \ c>0, \ D=1$
\begin{align*}
&u(x,t-\tau_2)=\phi(x+c(t-\tau_2)=  \phi(x+ct-c\tau_2)=\phi(\xi-c\tau_2)\\
&\frac{\partial u(x,t)}{\partial t}=c\phi'(x+ct)=c\phi'(\xi)\\
&\frac{\partial^2 u(x,t-\tau_1)}{\partial x^2}=\phi''(x+c(t-\tau_1))=\phi''(\xi-c\tau_1).
\end{align*}
Taking $r_1=c\tau_1, \ r_2=c\tau_2,$ equation \ref{Classmodel} becomes 
\begin{align*}
 c\phi'(\xi)&= \phi''(\xi-r_1)+ \phi(\xi-r_2)\left(1-\phi(\xi)\right)\\
 &= \phi''(\xi-r_1)-c\phi'(\xi)+ \phi(\xi-r_2)\left(1-\phi(\xi)\right)
\end{align*}
For simplicity we will write $t=\xi-r_1$, we have an equation of the form
\begin{equation} \label{cwm}
x''(t)-cx'(t+r_1)+x(t+(r_1-r_2))\left(1-x(t+r_1)\right)=0.      
\end{equation}
Define $f(\varphi)=\varphi(-r_2)\left(1-\varphi(0)\right),$ then $\varphi$ satisfies \[f_c(\varphi)-f_c(\psi ) +\beta (\varphi(0)-\psi (0)) \ge 0,\]
when $\psi\leq \varphi, \ \beta \geq 1.$ 
\subsubsection{Upper Solutions}
The quadratic function $P(\mu)=-\mu^2+c\mu -1$ has two positive solutions 
\[0< \mu_1=\frac{c-\sqrt{c^2-4}}{2}<\mu_2=\frac{c+\sqrt{c^2-4}}{2}.\]
\begin{proposition}
Let $c>2, \ 0<\theta<1, \ \mu_1$ as defined above and $r_i>0, i=1,2,$ are sufficiently small, then for all $t\in \R$   
\[\bar{\varphi}(t)=\frac{1}{1+\theta e^{-\mu_1 t}}\] is an upper solution of equation \ref{cwm} that belongs to $\Gamma.$
\end{proposition}
\begin{proof}
The proof is similar to that in Wu and Zou \cite{wuzou} with slight modifications. It is easy to show $\bar{\varphi}(t)\in \Gamma.$ In fact, $\bar{\varphi}(t)$ is differentiable for all $t\in \R,$ and by direct calculation 
\[\bar{\varphi}'(t)=\frac{\theta \mu_1 e^{-\mu_1 t}}{(1+\theta e^{-\mu_1 t})^2}>0.\]
Thus, $\bar{\varphi}(t)$ is non decreasing. We now turn our attention to the asymptotic behavior of $\bar{\varphi}(t).$ Indeed, 
\[\lim_{t\to -\infty}\bar{\varphi}(t)=\lim_{t\to -\infty}\frac{1}{1+\theta e^{-\mu_1 t}}=0,\]
and 
\[\lim_{t\to \infty}\bar{\varphi}(t)=\lim_{t\to \infty}\frac{1}{1+\theta e^{-\mu_1 t}}=1.\]
Therefore, $\bar{\varphi}(t)\in \Gamma.$  In order to show that $\bar{\varphi}(t)$ is an upper solution we first need to find $\bar{\varphi}'(t), \bar{\varphi}''(t)$.  
\[\bar{\varphi}'(t)=\frac{\theta \mu_1e^{-\mu_1 t}}{\left(1+\theta e^{-\mu_1 t}\right)^2}, \ \bar{\varphi}''(t)=\frac{\theta \mu_1^2 e^{-\mu_1 t}\left(\theta e^{-\mu_1 t}-1\right)}{\left(1+\theta e^{-\mu_1 t}\right)^3}.\]
It is easy to see that
\[\sup_{t\in \R}|\bar{\varphi}(t)|<\infty, \ \sup_{t\in \R}|\bar{\varphi}'(t)|<\infty.\]
It is easy to see that $\varphi''(t)$ is integrable almost everywhere .The only piece that is left to prove is \[D\bar{\varphi}''(t)-c\bar{\varphi}'(t+r_1)+\bar{\varphi}(t+r_1-r_2)\left(1-\bar{\varphi}(t+r_1)\right)\le 0.\] 
The first and second order derivatives without delay were calculated above. Now, 
\[\bar{\varphi}(t+r_1)=\frac{1}{1+\theta e^{-\mu_1 (t+r_1)}}, \ \bar{\varphi}(t+r_1-r_2)=\frac{1}{1+\theta e^{-\mu_1 (t+r_1-r_2)}}, \  \bar{\varphi}'(t+r_1)= \frac{\theta \mu_1e^{-\mu_1 (t+r_1)}}{\left(1+\theta e^{-\mu_1 (t+r_1)}\right)^2}.\]
Substituting into the wave model and performing the following calculation
\begin{align*}
  &\bar{\varphi}''(t)-c\bar{\varphi}'(t+r_1)+\bar{\varphi}(t+r_1-r_2)\left(1-\bar{\varphi}(t+r_1)\right)\\
  &= \frac{\mu_1^2 e^{-\mu_1 t}\left(\theta e^{-\mu_1 t}-1\right)}{\left(1+\theta e^{-\mu_1 t}\right)^3}-c \frac{\theta \mu_1e^{-\mu_1 (t+r_1)}}{\left(1+\theta e^{-\mu_1 (t+r_1)}\right)^2}+\frac{1}{1+\theta e^{-\mu_1 (t+r_1-r_2)}}\left(1-\frac{1}{1+\theta e^{-\mu_1 (t+r_1)}}\right).
 \end{align*}
  We are concerned with some small delay. In fact, 
  \begin{align*}
 &\lim_{r_1,r_2\to 0} \bar{\varphi}''(t)-c\bar{\varphi}'(t+r)+\bar{\varphi}(t+r)\left(1-\bar{\varphi}(t)\right)\\
 &=\frac{\theta e^{-\mu_1 t}}{\left(1+\theta e^{-\mu_1 t}\right)^2}\left(\frac{2D\theta \mu_1^2}{1+\theta e^{-\mu_1 t}}-\mu_1^2-c\mu_1+1\right)\\
 &=-\frac{\theta e^{-\mu_1 t}}{\left(1+\theta e^{-\mu_1 t}\right)^3}\left(-\theta\left(\mu_1^2+-c\mu_1+1\right)e^{-\mu_1 t}+ \mu_1^2 +c\mu_1-1\right).
\end{align*}. 
Notice that 
\[-\frac{\theta e^{-\mu_1 t}}{\left(1+\theta e^{-\mu_1 t}\right)^3}<0, \ t\in \R, \mu_1^2-c\mu_1+1=0.\]
Employing the fact that $\mu_1^2=c\mu_1-1,$ and $c\mu_1\in (1,2)$ when $c>2.$ The latter fact was shown in \cite{wuzou}. Thus, we can find some $r^*(c)>0$ such that when $0<r_i<r^*(c), \ i=1,2$, we have
\[\bar{\varphi}''(t)-c\bar{\varphi}'(t+r_1)+\bar{\varphi}(t+r_1-r_2)\left(1-\bar{\varphi}(t+r_1)\right)\leq-\frac{2\theta e^{-\mu_1 t}}{\left(1+\theta e^{-\mu_1 t}\right)^3}\left(c\mu_1-1\right)<0.\]
\end{proof} 

\subsubsection{Quasi-Lower Solutions}
The construction of lower solutions be be similar to the construction from the Belousov-Zhabotinskii Equations above. Indeed, we define 
$$\underline{\varphi}(t)=
\underline{ \varphi_1}(t)=\begin{cases} \frac{ e^{\lambda_1 t}}{k}, \ t\le 0 ,\\
\frac{1}{k} , \ t> 0.
\end{cases}
$$, 
\begin{proposition}
Let $0<r_2<r_1$ be small, $k\ge2$ and $c>2$, then take $\underline{\varphi}_1(t), \bar{\varphi}(t)$  as above, then  
\begin{enumerate}
\item $0<\underline{\varphi}(t)\leq \bar{\varphi}(t) \le 1$ for all $t\in \R$
\item $\underline{\varphi}(t)$ is a subsolution of equation (\ref{cwm}).
\end{enumerate}
\end{proposition}

\begin{proof}
The proof for part $i.)$ is easy to show, hence is omitted. For part $ii.)$ we proceed in cases like before. 

\noindent{\textit{Case 1}:} When $t\le - r_1$, we have the following since $r_1, r_2$ are sufficiently small 
$$
\underline{\varphi}^{\prime\prime}(t)-c\underline{\varphi}^{\prime}(t+r_1)+\frac{1}{2}  \underline{\varphi}(t+r_1-r_2)=0.
$$
Thus,
\begin{align*}
& \underline{\varphi}^{\prime\prime}(t)-c\underline{\varphi}^{\prime}(t+r_1)+\underline{\varphi}(t+r_1-r_2)\left(
1-\underline{\varphi}(t+r_1)\right) \\
&=\underline{\varphi}(t+r_1-r_2)\left(
\frac{1}{2}-\underline{\varphi}(t+r_1)\right)>0,
\end{align*}
since $\underline{\varphi}(t+r_1)=\frac{ e^{\lambda_1 t}}{k}\le \frac{1}{2}.$

\medskip
\noindent{\textit{Case 2}:$-r_1\le t\le 0.$}
There are technically two cases. The first is $t+r_1-r_2\le 0$ and the second being $t+r_1-r_2>0$. The cases as shown in the same manner, so we only show the first case. Indeed, the  equation becomes
\begin{align*}
 &B:=\underline{\varphi}^{\prime\prime}(t)-c\underline{\varphi}^{\prime}(t+r_1)+\underline{\varphi}(t+r_1-r_2)\left(
(1-\underline{\varphi}(t+r_1)\right) \\
&=\underline{\varphi}^{\prime\prime}(t)+\underline{\varphi}(t+r_1-r_2)\left(
(1-\underline{\varphi}(t+r_1)\right)>0,
\end{align*}
since $\underline{\varphi}^{\prime\prime}(t)>0, 1-\underline{\varphi}(t+r_1)\ge 0.$

\medskip
\noindent{\textit{Case 3}:$  0\le t.$} Since, $\underline{\varphi}(t)=\frac{1}{k}$ it is easy to see

\begin{align*}
 &B:=\underline{\varphi}^{\prime\prime}(t)-c\underline{\varphi}^{\prime}(t+r_1)+\underline{\varphi}(t+r_1-r_2)\left(
(1-\underline{\varphi}(t+r_1)\right) \\
&=\underline{\varphi}(t+r_1-r_2)\left(
(1-\underline{\varphi}(t+r_1)\right)>0.
\end{align*}
This completes the proof.
\end{proof}

\medskip \noindent
So, we have the following.
\begin{corollary}
Assume that $c>2$ is given. Then, Eq.(\ref{cwm}) has a traveling wave solution $u(x,t)=\phi (x+ct)$ for sufficiently delays $\tau_1,\tau_2$.
\end{corollary}

\section*{Declarations}
\subsection*{Ethical Approval}
Ethical approval is not applicable to our research in this manuscript because it does not involve human or animal studies.
\subsection*{Funding}
This manuscript has no funding.

\subsection*{Availability of data and materials }
This manuscript has no associated data and materials.

\bibliographystyle{amsplain}

\end{document}